\documentclass[11pt]{amsart}
\usepackage{amsmath,amssymb,amsthm,latexsym,cite,cancel}
\usepackage[small]{caption}
\usepackage{graphicx,wasysym,overpic,tikz,color}
\usepackage{subfigure,color}
\usepackage{cite}
\usepackage[colorlinks=true,urlcolor=blue,
citecolor=red,linkcolor=blue,linktocpage,pdfpagelabels,
bookmarksnumbered,bookmarksopen]{hyperref}
\usepackage[italian,english]{babel}
\usepackage{units}
\usepackage{enumitem}
\usepackage[left=2.1cm,right=2.1cm,top=2.71cm,bottom=2.71cm]{geometry}
\usepackage[hyperpageref]{backref}

\usepackage[colorinlistoftodos]{todonotes}

\usepackage{amsfonts}
\usepackage{amscd}
\usepackage{latexsym}
\usepackage{mathrsfs}

\newtheorem{theorem}{Theorem}[section]

\newtheorem{lemma}[theorem]{Lemma}

\newtheorem{remark}[theorem]{Remark}

\numberwithin{equation}{section}

\title[The generalized Kadomtsev-Petviashvili equation]{Existence and multiplicity of normalized solutions for the generalized Kadomtsev-Petviashvili equation in $\mathbb{R}^2$}

\author[C.O. Alves]{Claudianor O. Alves}

\address[C.O. Alves]{\newline\indent
    Unidade Acad\^{e}mica de Matem\'{a}tica
    \newline\indent
    Universidade Federal de Campina Grande
    \newline\indent
    PB CEP:58429-900, Brazil}
\email{\href{mailto: coalves@mat.ufcg.edu.br}{coalves@mat.ufcg.edu.br}}

\author[R. Ding]{Rui Ding}

\address[R. Ding]{\newline\indent
    School of Mathematics
    \newline\indent
    East China University of Science and Technology
    \newline\indent
    Shanghai 200237, PR China }
\email{\href{mailto:dingrui18363922468@outlook.com }{dingrui18363922468@outlook.com }}

\author[C. Ji]{Chao Ji}

\address[C. Ji]{\newline\indent
    School of Mathematics
    \newline\indent
    East China University of Science and Technology
    \newline\indent
    Shanghai 200237, PR China }
\email{\href{mailto:jichao@ecust.edu.cn}{jichao@ecust.edu.cn}}

\newcommand\R{\mathbb R}
\newcommand\N{\mathbb N}

\subjclass[2020]{35A15, 35A18.}
\date{\today}

\begin{document}
\maketitle

    \begin{abstract}
        In this paper, we study the existence and {multiplicity} of nontrivial solitary waves for the generalized Kadomtsev-Petviashvili equation with prescribed {$L^2$-norm}
\begin{equation*}\label{Equation1}
    \left\{\begin{array}{l}
        \left(-u_{x x}+D_x^{-2} u_{y y}+\lambda u-f(u)\right)_x=0,{\quad x \in \mathbb{R}^2, } \\[10pt]
    \displaystyle   \int_{\mathbb{R}^2}u^2 d x=a^2,
    \end{array}\right.
\end{equation*}
where  $a>0$ and  $\lambda \in \mathbb{R}$ is an unknown parameter that appears as a Lagrange multiplier. For the case $f(t)=|t|^{q-2}t$, with $2<q<\frac{10}{3}$ ($L^2$-subcritical case) and $\frac{10}{3}<q<6$ ($L^2$-supercritical case), we establish the existence of normalized ground state solutions for the above equation. Moreover, when $f(t)=\mu|t|^{q-2}t+|t|^{p-2}t$, with $2<q<\frac{10}{3}<p<6$ and $\mu>0$,
we prove the existence of normalized ground state solutions  which corresponds to a local minimum of the associated energy
functional. In this case, we further show that there exists a sequence $(a_n) \subset (0,a_0)$ with $a_n \to 0$ as $n \to+\infty$, such that for each $a=a_n$, the problem admits a second solution with positive energy.  To the best of our knowledge, this is the first work that studies the existence of solutions for the generalized Kadomtsev-Petviashvili equations under the $L^2$-constraint, which we refer to them as the normalized solutions.
\end{abstract}
\thispagestyle{empty}

\noindent
\section{Introduction}
This paper concerns the existence and multiplicity of nontrivial solitary waves to the following
generalized Kadomtsev-Petviashvili equation with a prescribed {$L^2$-norm}
\begin{equation}\label{Equation}
    \left\{\begin{array}{l}
        \left(-u_{x x}+D_x^{-2} u_{y y}+\lambda u-f(u)\right)_x=0,{\quad (x,y) \in \mathbb{R}^2, } \\[10pt]
    \displaystyle   \int_{\mathbb{R}^2}|u|^2 d x=a^2,
    \end{array}\right.
\end{equation}
where $a>0$, $\lambda \in \mathbb{R}$ is unknown and will arise as a Lagrange multiplier and $D_x^{-1} h(x, y)=\int_{-\infty}^x h(s, y) d s$. Throughout this paper, we refer to solutions of this type as normalized solutions. Equation \eqref{Equation} arises when seeking solutions with a prescribed {$L^2$-norm} for the generalized Kadomtsev-Petviashvili equation
\begin{equation}\label{GKPN=2}
    \psi_t+\psi_{x x x}+(f(\psi))_x=D_x^{-1} \psi_{yy},
\end{equation}
where $(t, x, y) \in \mathbb{R}^{+} \times \mathbb{R} \times \mathbb{R}$. {This equation is a special case of the following equation
\begin{equation}\label{GKP}
        \psi_t+\psi_{x x x}+(f(\psi))_x=D_x^{-1} \Delta_y \psi,
\end{equation}
where $(t, x, y) \in \mathbb{R}^{+} \times \mathbb{R} \times \mathbb{R}^{N-1}$ and $\Delta_y=\sum_{i=1}^{N-1} \frac{\partial^2}{\partial y_i^2}$, $N \geq 2$. When $N=2$, equation \eqref{GKP} reduces to \eqref{GKPN=2}.

We are interested, in particular, in the existence of a {\it solitary wave} for \eqref{GKP}.
A solitary wave is a solution of the form {$\psi(t, x, y) = u(x-\lambda t, y)$, where $\lambda \in \mathbb{R}$ and $u: \mathbb{R}^N \rightarrow \mathbb{R}$ is a time-independent function}.
 Substituting in \eqref{GKP}, we obtain
$$
u_{x x x}-\lambda u_x+(f(u))_x=D_x^{-1} \Delta_y u, \quad x\in \mathbb{R}^N,
$$
or, equivalently,
\begin{equation}\label{Equation-l}
\left(-u_{x x}+D_x^{-2} \Delta_y u+\lambda u-f(u)\right)_x=0, \quad x\in \mathbb{R}^N.
\end{equation}

A possible choice is then to fix $\lambda \in \mathbb{R}$, and to search for solutions to \eqref{Equation-l}.
For this class of problems, we first mention the pioneering work
due to  De Bouard and Saut \cite{DS0, DS}, which treated a nonlinearity {$f(t)=\frac{1}{p+1} t^{p+1}$} assuming that $p=\frac{m}{n}$, with $m$ and $n$ relatively prime, and $n$ is odd and $1 \leq p<4$, if $N=2$, or $1 \leq p<\frac{4}{3}$, if $N=3$. In \cite{DS0, DS}, De Bouard and Saut obtained existence and nonexistence of solitary waves by using constrained minimization method and concentration-compactness principle \cite{LionsCCPLocal2}. Willem \cite{Willem} extended the results of \cite{DS0} to the case \( N = 2 \), where \( f(u) \) is a continuous function satisfying certain structural conditions, by applying the mountain pass theorem.
Then Wang and Willem \cite{WW} obtained multiple solitary waves in one-dimensional spaces via the Lyusternik-Schnirelman category theory.  In \cite{LS}, by using the variational method, Liang and Su proved the existence of nontrivial solitary waves for equation \eqref{Equation-l} with
$f(x, y, u)=Q(x, y)\vert u\vert^{p-2}u$ and $N\geq 2$, where $Q\in C(\mathbb{R}\times\mathbb{R}^{N-1}, \mathbb{R})$ satisfying some assumptions and $2<p<N^{*}=\frac{2(2N-1)}{2N-3}$, while Xuan \cite{Xu} dealt with the case where $f(u)$ satisfies some superlinear conditions in higher dimension. By applying the linking theorem established in \cite{SZ}, He and Zou \cite{HZ} studied existence of nontrivial solitary waves for the generalized Kadomtsev-Petviashvili equation in multi-dimensional spaces. Similar results can also be found in \cite{Z}.

Recently, Alves and Miyagaki in \cite{AM} established some results concerning the existence, regularity and concentration
phenomenon of nontrivial solitary waves for a class of generalized variable
coefficient Kadomtsev-Petviashvili equations in $\mathbb{R}^2$. After that, Figueiredo and Montenegro \cite{FM} proved the existence of multiple solitary waves for a generalized Kadomtsev-Petviashvili equation with a potential in $\mathbb{R}^2$. In \cite{FM}, the authors showed that the number of solitary waves corresponds to the number of global minimum points of the potential when a parameter is small enough. {In \cite{AJ}, Alves and Ji studied the existence and concentration of the nontrivial solitary waves for the generalized Kadomtsev-Petviashvili equation in $\mathbb{R}^2$ when the potential satisfies a local assumption due to  del Pino and Felmer \cite{DF}. For further results on the generalized Kadomtsev-Petviashvili equations,   we refer the reader to }\cite{AMP, AGM, BIN, B, F, GW, HNS, IM, KP1, S, SZ, TG, T, WW, XWD, ZXM} and the references therein.
While most existing works focused on solutions to \eqref{Equation-l} with fixed \( \lambda \in \mathbb{R} \), it is also meaningful to consider solutions to \eqref{Equation-l} with a prescribed \( L^2 \)-norm. In this setting, the parameter \( \lambda \) is unknown and arises as a Lagrange multiplier.
    We refer to such solutions as \emph{normalized solutions}. To the best of our knowledge, this is the first work that studies normalized solutions for the generalized Kadomtsev--Petviashvili equation. In particular, we focus on the case $N=2$ since the Pohozaev identity plays a crucial role in our analysis, while for $N\geq3$, whether the solutions satisfy the Pohozaev identity depends on their regularity, which is difficult to verify.
From a physical point of view, we would like to point out that solutions \( \psi(t,x,y) \) to \eqref{GKPN=2} satisfy the conservation of momentum (see \cite{SYM}, for instance), specifically,
    \[
    |\psi(t,\cdot,\cdot)|_2^2 = \int_{\mathbb{R}^2} |\psi(0,\cdot,\cdot)|^2\, dx\,dy, \quad \forall t \in \mathbb{R}.
    \]
When $f(t) = \frac{1}{2}t^2$, equation \eqref{GKPN=2} reduces to the Kadomtsev--Petviashvili I(KP-I), which models the propagation of weakly nonlinear dispersive long waves on the surface of a fluid, where the wave motion is predominantly one-dimensional with weak transverse effects along the $y$-axis (see \cite{KP1, PetviashviliYankov}).
    In this context, the momentum is conserved, reflecting a fundamental physical law.
This naturally leads to the important question of whether equation \eqref{Equation} admits solutions with a prescribed
$L^2$-norm, which can be found by seeking critical points of the associated energy functional
\begin{equation*}
    I(u)=\frac{1}{2}\int_{\mathbb{R}^{2}} \left(|u_{x}|^{2} + |D^{-1}_{x} u_{y}|^{2} \right)dx dy-\int_{\mathbb{R}^2} F(u) \,dxdy
\end{equation*}
on the constraint
$$
S(a)=\left\{u\in X, \int_{\mathbb{R}^{2}}|u|^{2}dxdy=a^{2}\, \right\}
$$
for the definition of space $X$, please see Section \ref{sec2}. Over the past decades, the normalized solutions have been extensively studied for the following nonlinear Schr\"{o}dinger equation
\begin{equation}\label{Equation-e}
    \left\{\begin{array}{l}
        -\Delta u=\lambda u+g(u), \quad x\in \mathbb{R}^N, \\[5pt]
    \displaystyle \int_{\mathbb{R}^N}|u|^2 d x=a^2.
    \end{array}\right.
\end{equation}
Equation \eqref{Equation-e} arises when one looks for solutions with prescribed {$L^2$-norm} for the nonlinear Schr\"{o}dinger equation
\begin{equation}\label{NOS}
i \frac{\partial \psi}{\partial t}+\Delta \psi+h\left(|\psi|^2\right) \psi=0, \quad x\in \mathbb{R}^N,
\end{equation}
where $h\left(|u|^2\right) u=g(u)$.
A stationary wave solution is a solution of the form $\psi(t, x)=e^{-i \lambda t} u(x)$, where $\lambda \in \mathbb{R}$ and $u: \mathbb{R}^N \rightarrow \mathbb{R}$ is a time-independent function that must solve the elliptic problem
\begin{equation}\label{Equation-e1}
-\Delta u=\lambda u+g(u), \quad x\in \mathbb{R}^N.
\end{equation}
 For some values of $\lambda$,  the existence of nontrivial solutions for \eqref{Equation-e1} are obtained as the critical points of the action functional $E_\lambda: H^1\left(\mathbb{R}^N\right) \rightarrow \mathbb{R}$ given by
$$
E_\lambda(u)=\frac{1}{2} \int_{\mathbb{R}^N}\left(|\nabla u|^2-\lambda |u|^2\right) d x-\int_{\mathbb{R}^N} G(u) d x
$$
where $G(t)=\int_0^t g(s) d s$.
Another important way to find the nontrivial solutions for \eqref{Equation-e1} is to search for solutions with prescribed $L^2$-norm, in which case $\lambda \in \mathbb{R}$ appears as part of the unknown. This approach seems to be particularly meaningful from the physical point of view, since it gives a better insight of the properties of the stationary solutions for \eqref{NOS}, such as stability or instability(see \cite{CL}). In particular, when focusing on the case $g(u)=|u|^{q-2} u$, the associated energy functional is given by
\begin{equation}\label{Eq-Eu}
E(u)=\frac{1}{2} \int_{\mathbb{R}^N}|\nabla u|^2d x-\int_{\mathbb{R}^N} |u|^qd x.
\end{equation}
We recall that the $L^2$-critical exponent
$\bar{q}:= 2 + \frac{4}{N}$, which is generated from the Gagliardo-Nirenberg inequality (see \cite{Cazenave}) and plays a special role.
In the $L^2$-subcritical case, namely $q\in(2,\bar{q})$, the corresponding energy functional on the constraint is coercive and bounded from below. Hence one can obtain the existence of a global minimizer by minimizing on the $L^2$-sphere, {cf.} \cite{Shibata}. In the $L^2$-supcritical case, that is $q\in(\bar{q},2^*)$, the functional on the $L^2$-sphere could not be bounded from below. {To prove the boundedness of the corresponding Palais-Smale sequence,} Jeanjean \cite{Jean} employed a mountain pass structure for an auxiliary functional. More precisely, he studied the following equation
$$
\left\{\begin{array}{l}
    -\Delta u=\lambda u+g(u), \quad x\in \mathbb{R}^N, \\[5pt]
\displaystyle   \int_{\mathbb{R}^N}|u|^2 d x=a^2,
\end{array}\right.
$$
with $N \geq 2$, where function $g: \mathbb{R} \rightarrow \mathbb{R}$ is an odd continuous function with subcritical growth that satisfies some technical conditions. One of these conditions is the following: $\exists(\alpha, \beta) \in \mathbb{R} \times \mathbb{R}$ satisfying
$$
\begin{cases}\frac{2 N+4}{N}<\alpha \leq \beta<\frac{2 N}{N-2}, & \text { for } N \geq 3, \\ \frac{2 N+4}{N}<\alpha \leq \beta, & \text { for } N=1,2,\end{cases}
$$
such that
$$
\alpha G(s) \leq g(s) s \leq \beta G(s) \quad \text { with } \quad G(s)=\int_0^s g(t) d t.
$$
An example of a function $g$ that satisfies the above condition is $g(s)=|s|^{q-2} s$ with $q \in\left(\bar{q}, 2^*\right)$ for $N \geq 3$. To address the lack of compactness of the Sobolev embedding on the whole space $\mathbb{R}^N$, the author worked within the radially symmetric subspace $H_{\text{rad}}^1\left(\mathbb{R}^N\right)$ to recover some compactness. For further results on normalized solutions of nonlinear Schr\"{o}dinger equations, we refer the reader to \cite{A21, AS2, CCM, AJ21, AlvesThin, Bartschmolle, BartschSaove,ThomaseNicola, valerio, Bartosz, BellazziniJeanjeanLuo, Cazenave, Jun, JeanjeanJendrejLeVisciglia, JeanjeanLu, JeanjeanLu2020,JeanjenaLu2, JeanjeanLe, Nicola1} and the references therein.

Inspired by the discussions above, we will first establish a Gagliardo-Nirenberg inequality suitable for the generalized Kadomtsev--Petviashvili equation (see Lemma \ref{L1} in Section \ref{sec2}). This inequality plays a crucial role in the variational analysis of normalized solutions for generalized Kadomtsev-Petviashvili equation in $\mathbb{R}^2$. Based on the Gagliardo-Nirenberg inequality and {setting the following special scaling}
    \[
    \mathcal{H}(u,t) = e^{t} u(e^{\frac{2}{3}t}x, e^{\frac{4}{3}t}y).
    \]
In analogy with the nonlinear Schr\"{o}dinger equations, when $ f(t)=|t|^{q-2}t $ in \eqref{Equation}, we define the $L^2$-critical exponent
for  Kadomtsev-Petviashvili equations in $\mathbb{R}^2$ as $q=\frac{10}{3}$. The case $q \in (2, \frac{10}{3})$ corresponds to the $L^2$-subcritical case, in which the energy functional under the $L^2$-constraint is bounded from below, and the case $q \in (\frac{10}{3}, 6)$ corresponds to  the $L^2$-supercritical case, where the energy functional under the $L^2$-constraint is unbounded from below.

{One of the main difficulties in studying normalized solutions of the generalized Kadomtsev-Petviashvili equations in $\mathbb{R}^2$ is the lack of the compactness. For the nonlinear Schr\"{o}dinger equations, this difficulty is often overcome by working within the radially symmetric subspace $H_{\text{rad}}^1\left(\mathbb{R}^N\right)$, which helps recover compactness (see \cite{Jean}). However, this approach is not applicable to the generalized Kadomtsev-Petviashvili equations in $\mathbb{R}^2$   by the following facts:
\begin{itemize}
	\item[(i)] If  $X_{\text{rad}}$ denotes the subspace of functions in $X$ that are radially symmetric with respect to 0, it is unclear whether  the embedding $X_{\text{rad}} \hookrightarrow L^{p}(\mathbb{R}^2)$ is compact for $p \in (2,6)$.
	\item[(ii)] If $u \in X_{\text{rad}}$, the scaling function $\mathcal{H}(u, t)(x, y) = e^{t}u(e^{\frac{2}{3} t}x, e^{\frac{4}{3} t}y)$ may not be radially symmetric.
	\item[(iii)] If $u \in X$ and $g \in O(2)$, we don't know if $\|g.u\|=\|u\|$, where $O(2)$ denotes  the orthogonal group of dimension 2,  $g.u(x,y):=u(g(x,y))$ for all $(x,y) \in \mathbb{R}^2$ and $\|\cdot \|$ is the norm of $X$ defined in Section 2. This property is crucial in order to apply the Palais' Principle of symmetric criticality, see Willem \cite[Chapter 1, Section 1.6]{Willem}.
\end{itemize}
To address this issue, we borrow the ideas developed by Jeanjean \cite{Jean} and provide some variational characterization of the mountain pass level.

{Our main results are as follows:}
\begin{theorem}\label{Th1}
    Assume that $ f(t)=|t|^{q-2}t $. If $2<q<\frac{10}{3}$, for any $a>0$, problem \eqref{Equation} admits a couple $\left({u}, \lambda\right) \in S(a) \times \R$ of weak solutions and $\lambda<0$. {In addition, ${u}$ is a normalized ground state solution of \eqref{Equation}}.
\end{theorem}
{\begin{theorem}\label{Th-1}
        Assume that $ f(t)=|t|^{q-2}t $. If $q=\frac{10}{3}$, there exists $a^*>0$,
        such that for $a\in(0,a^*]$, problem \eqref{Equation} has no nontrivial solution.
\end{theorem}}

\begin{theorem}\label{Th2}
    Assume that $ f(t)=|t|^{q-2}t $. If $\frac{10}{3}<q<6$, for any $a>0$, problem \eqref{Equation} admits a couple $\left({u}, \lambda\right) \in S(a) \times \R$ of weak solutions and $\lambda<0$. In addition, ${u}$ is a normalized ground state solution of \eqref{Equation}.
\end{theorem}

{\begin{remark}
When studying the $L^2$-critical case of the normalized solutions to the nonlinear Schr\"{o}dinger equation, that is, when \( q = 2 + \frac{4}{N} \) in \eqref{Eq-Eu}, there exists \( a_0 > 0 \) such that for any \( a > a_0 \),
\begin{equation}\label{Eq-ma}
    m(a):=\inf _{u \in \tilde{S}(a)} E(u)=-\infty
\end{equation}
 holds, where $\tilde{S}(a)=\left\{u\in H^1(\mathbb{R}^N),   \int_{\mathbb{R}^{N}}|u|^{2}dx=a^{2}\, \right\}$ and $ E(u) $ is defined \eqref{Eq-Eu}.
  The proof of \eqref{Eq-ma} relies on the attainability of the optimal function in the following Gagliardo-Nirenberg inequality
 {  \begin{equation}\label{Eq-GN inequality}
    |u|_q^q \leq {C}_{N, p}|\nabla u|_2^{2}|u|_2^{\frac{4}{N}}, \quad \forall u \in H^1(\mathbb{R}^N).
  \end{equation}}
However, since the attainability of the corresponding inequality in the generalized Kadomtsev--Petviashvili equation is still unknown (see \eqref{Gagliardo}), it remains unclear whether there exists \( a_0 > 0 \) such that the same conclusion holds.
\end{remark}
}

{
Recall that, when dealing with the nonlinear Schr\"{o}dinger equations,  the $L^2$-critical exponent
$$
 \bar{q}:= 2 + \frac{4}{N}
 $$
plays a special role. In \cite{SoavenonC}, Soave studied the nonlinear Schr\"{o}dinger equation with combined power nonlinearities, that is
\begin{equation}\label{Equation111}
    \left\{\begin{array}{l}
        -\Delta u=\lambda u+\mu|u|^{q-2} u+|u|^{p-2} u,{\quad x \in \mathbb{R}^N, } \\[10pt]
        \displaystyle   \int_{\mathbb{R}^N}|u|^2 d x=a^2,
    \end{array}\right.
\end{equation}
where $2<q \leq \bar{p} \leq p<2^*,$ with $p \neq q$ and $\mu \in \mathbb{R}$. As shown in \cite{SoavenonC}, the interplay between subcritical, critical and supercritical nonlinearities has a deep impact on the geometry of the functional and on the existence and properties of normalized ground state solutions. From some point of view, this can be considered as a kind of Br\'ezis-Nirenberg problem in the context of normalized solutions.
In particular, in the case where $2<q < \bar{p} < p<2^*$, it was proved that \eqref{Equation111}
admits a normalized ground state solution with negative energy,  as well as another solution of mountain pass type with positive energy.
Motivated by the research in \cite{SoavenonC}, in this paper we will also study the generalized Kadomtsev-Petviashvili equation with combined power nonlinearities, namely
$$ f(t)=\mu|t|^{q-2}t +|t|^{p-2}t$$
in \eqref{Equation}, where {$2<q<\frac{10}{3}<p<6$} and $\mu>0$.
However, since our approach does not allow us to work in the radially symmetric subspace, it is more difficult to overcome the lack of compactness. As a result, we are able to  prove the existence of normalized ground state solutions for the generalized Kadomtsev-Petviashvili equation in $\mathbb{R}^2$. Moreover, for a sequence $(a_n) \subset (0,a_0)$ with $a_n \to 0$ as $n \to+\infty$, we show that problem \eqref{Equation} with $a=a_n$ admits a second solution with positive energy.

\begin{theorem}\label{Th3}
     Assume that $ f(t)=\mu|t|^{q-2}t +|t|^{p-2}t$. If $2<q<\frac{10}{3}<p <6$, {there exists $a_0=a_0(\mu)>0$ such that}, for any $a \in\left(0, a_0\right)$,
     \eqref{Equation} admits a couple {$\left(u, \lambda\right) \in S(a) \times \R$ of weak solutions with  $\lambda<0$ and $u$ being a normalized ground state solution of \eqref{Equation}}. {Moreover, there is $(a_n) \subset (0,a_0)$ with $a_n \to 0$ as $n \to+\infty$, such that problem \eqref{Equation} with $a=a_n$ admits a second solution with positive energy.}
\end{theorem}
\begin{remark}
In \cite{SoavenonC}, Soave studied the normalized solutions for the nonlinear Schr\"{o}dinger equation with combined nonlinearities. Especially, for the case involving a mixture of $L^{2}$-subcritical and $L^{2}$-supercritical nonlinearities, it was shown that a second solution with positive energy exists for all $a\in (0, a^{*})$, where $a^{*}$ is a positive number, see \cite[Theorem 1.3]{SoavenonC} for details. To overcome the lack of compactness in proving the existence of this second solution, Soave worked in the subspace of radial functions. However, for problem \eqref{Equation} considered in this paper, we have no information about radial functions in the working space $X$. Therefore, instead of relying on radial symmetry, we apply a different approach and prove the existence of a second solution with positive energy for a sequence $(a_n) \subset (0,a_0)$ with $a_n \to 0$ as $n \to+\infty$. Whether such a result holds for every $a\in (0, a_{0})$ remains an interesting and open problem.
\end{remark}

\noindent\textbf{Notations:}
\noindent
For $1 \leq p < \infty$ and $u \in L^p(\mathbb{R}^N)$, we denote $|u|_p := \left(\int_{\mathbb{R}^N} |u|^p\, dx \right)^{1/p}$. We use ``$\to$'' and ``$\rightharpoonup$'' to denote strong and weak convergence in the corresponding function spaces, respectively. $C$ and $C_i$ denote positive constants. $\langle \cdot, \cdot \rangle_X$ and $\langle \cdot, \cdot \rangle_{\mathbb{R}}$ denote the inner products in $X$ and $\mathbb{R}$, respectively. $X^*$ denotes the dual space of $X$. Finally, $o_n(1)$ and $O_n(1)$ denote quantities satisfying $|o_n(1)| \to 0$ and $|O_n(1)| \leq C$ as $n \to \infty$, respectively.

\section{Functional setting}\label{sec2}
Arguing as in Willem \cite[Chapter 7]{Willem}, we define the inner product on the set $Y=\{ g_x : g \in C^{\infty}_0 (\mathbb{R}^{2}) \}$ by
\begin{eqnarray}\label{Produtointerno}
    \langle u,v\rangle_X= \displaystyle\int_{\mathbb{R}^{2}} \left(u_x v_x + D^{-1}_{x} u_{y} D^{-1}_{x} v_{y} + u v\right)dx dy
\end{eqnarray}
and the corresponding norm by
\begin{eqnarray}\label{Norma}
    \|u\|= \left( \displaystyle\int_{\mathbb{R}^{2}} \left(|u_{x}|^{2} + |D^{-1}_{x} u_{y}|^{2} + |u|^{2}\right)dx dy\right)^{\frac{1}{2}}.
\end{eqnarray}
In the sequel, let us proceed to define the function space $X$ in which  we will work. A function $u:\mathbb{R}^{2}\to\mathbb{R}$ belongs to $X$ if there exists a sequence $(u_n) \subset Y$ such that
$$
u_n \to u \ \  a.e. \ \ \text{in}\ \ \mathbb{R}^{2} \ \ \mbox{and} \ \ \|u_j -u_k\|\to 0 \ \ \mbox{as} \ \ j, k \to +\infty.
$$
The space $X$ with inner product (\ref{Produtointerno}) and norm (\ref{Norma}) is a Hilbert space.
Hereafter, we also define
\begin{eqnarray}\label{Produtointerno0}
    \langle u,v \rangle_0= \displaystyle\int_{\mathbb{R}^{2}} \left(u_x v_x + D^{-1}_{x} u_{y} D^{-1}_{x} v_{y}\right)dx dy, \quad \forall u,v \in X
\end{eqnarray}
and
\begin{equation}\label{Norma0}
    \|u\|_0=  \left(\displaystyle\int_{\mathbb{R}^{2}} \left(|u_{x}|^{2} + |D^{-1}_{x} u_{y}|^{2} \right)dx dy\right)^{\frac{1}{2}}, \quad \forall u \in X.
\end{equation}
From \cite[Lemma 2.2]{Xu} , we know that  there exists a constant $S>0$ such that
\begin{equation} \label{S}
    |u|_6 \leq S\left(\int_{\mathbb{R}^2}\left(\left|u_x\right|^2+\left|D_x^{-1} u_y \right|^2\right) d xdy\right)^{\frac{1}{2}}, \quad \forall u \in X.
\end{equation}
We say that $(u, \lambda)$ is a solution  to \eqref{Equation} if $u \in X$ satisfies $\displaystyle \int_{\mathbb{R}^{2}}|u|^{2}dx=a^{2}$, $\lambda\in \mathbb{R}$, and
$$
\langle u,\phi\rangle_0 -\lambda\int_{\mathbb{R}^{2}}u\phi dx dy- \displaystyle\int_{\mathbb{R}^{2}} f(u)\phi dx dy = 0  \,\,\,\, \mbox{for all} \,\,  \phi\in X.
$$
Hereafter,  we say that $u \in X$ is a {\it normalized ground state solution}, if there is $\lambda \in \mathbb{R}$ such that $(u,\lambda)$ is a solution of \eqref{Equation}, and $u$ has minimal energy among all solutions which belong to $S(a)$, that is
$$
\left(\left.J\right|_{S(a)}\right)^{\prime}(u)=0 \quad \text { and } \quad J(u)=\inf \left\{J(w):\left(\left.J\right|_{S(a)}\right)^{\prime}(w)=0 \text { and } w \in S(a)\right\},
$$
{where $J: X \rightarrow \R$ is the functional defined by
\begin{equation*}
    J(u)=\frac{1}{2}\int_{\mathbb{R}^{2}} \left(|u_{x}|^{2} + |D^{-1}_{x} u_{y}|^{2} \right)dx dy-\int_{\mathbb{R}^2} F(u) \,dxdy.
\end{equation*}}

{Next, we establish a Gagliardo-Nirenberg type inequality, which is fundamental for the study of normalized solutions of generalized Kadomtsev-Petviashvili equations in $\mathbb{R}^2$.}
\begin{lemma} \label{L1}(The Gagliardo-Nirenberg inequality)
    For any $q\in [2, 6]$, we have
{   \begin{equation}\label{Gagliardo}
        |u|^{q}_{q} \leq C_q|u|^{(1-\beta )q}_{2}\left(\int_{\mathbb{R}^2}\left(\left|u_x\right|^2+\left|D_x^{-1} u_y\right|^2\right) d xdy\right)^{\frac{q\beta}{2}} \quad \forall u \in X,
    \end{equation}}
where  $ \beta=\frac{3}{2}-\frac{3}{q}$, for some positive constant $C_q>0.$
\end{lemma}
\begin{proof}
    If $q=2$ or $q=6$, the above inequality is straightforward. For $q\in (2, 6)$, we apply the interpolation inequality of Lebesgue's space combined with \eqref{S}. Then,
    \begin{align*}
        |u|^{q}_q \leq & |u|^{q(1-\beta)}_2|u|^{q\beta}_6 \\
        \leq  & S^{q\beta}|u|^{q(1-\beta)}_2\left(\int_{\mathbb{R}^2}\left(\left|u_x\right|^2+\left|D_x^{-1} u_y\right|^2\right) d xdy\right)^{\frac{q\beta}{2}}\\
        =  & C_q|u|^{q(1-\beta)}_2\left(\int_{\mathbb{R}^2}\left(\left|u_x\right|^2+\left|D_x^{-1} u_y\right|^2\right) d xdy\right)^{\frac{q\beta}{2}},
    \end{align*}
    where $C_q=S^{q\beta}>0$  depends only on $q$.
\end{proof}
\section{{$L^2$-subcritical case}}
In this section, without further mention, we assume that {$f(t)=|t|^{q-2}t$} and $q \in (2,\frac{10}{3})$. {Our main goal is to study} the existence of a minimizer for the $L^{2}$-constraint minimization problem:}
$$
\Upsilon_a=\inf_{u \in S(a)}J(u).
$$
{First of all, we show that the functional $J$ is coercive on $S(a)$ using the Gagliardo-Nirenberg inequality established in Lemma \ref{L1}.
\begin{lemma}\label{L2}
For any $a>0$, the functional $J$ is coercive on $S(a)$.
\end{lemma}

\begin{proof}
By Lemma \ref{L1}, we have
\begin{equation}
\begin{aligned}
J(u)&=\frac{1}{2}\int_{\mathbb{R}^{2}} \left(|u_{x}|^{2} + |D^{-1}_{x} u_{y}|^{2} \right)dx dy-\frac{1}{q}\int_{\mathbb{R}^2} |u|^q \,dxdy\\
&\geq\frac{1}{2}\int_{\mathbb{R}^{2}} \left(|u_{x}|^{2} + |D^{-1}_{x} u_{y}|^{2} \right)dx dy-\frac{1}{q}C_qa^{q(1-\beta)}\left(\int_{\mathbb{R}^2}\left(\left|u_x\right|^2+\left|D_x^{-1} u_y\right|^2\right) d xdy\right)^{\frac{q\beta}{2}}.
\end{aligned}
\end{equation}
Since $q\in (2, \frac{10}{3})$, it follows that ${\frac{q\beta}{2}}<2$. Hence, $J$ is bounded from below and coercive on $S(a)$.
\end{proof}
{Our next lemma shows that $\Upsilon_a<0$, a very useful property for ruling out the vanishing for the minimizing sequence of $\Upsilon_a$.}
\begin{lemma}\label{L12}
    For any $a>0$, it holds that $-\infty<\Upsilon_a<0$.
\end{lemma}
\begin{proof}
    Let $u\in S(a)$ and set $u_t=e^{t}u(e^{\frac{2}{3}t}x,e^{\frac{4}{3}t}y) $ for $t\in\R$, which satisfies $u_t\in S(a)$. {A direct computation yields}
    $$
    J(u_t)=\frac{e^{\frac{4}{3}t}}{2}\int_{\mathbb{R}^2}\left(|u_x|^2+|D^{-1}_{x} u_{y}|^2\right)\,dxdy-\frac{e^{(q-2)t}}{q}\int_{\mathbb{R}^2}|u|^q\,dxdy.
    $$
    Since $q\in(2,\frac{10}{3})$, it follows that $J(u_t)<0$ for $t<0$ and $|t|$ large enough, from where it follows that \linebreak  $\Upsilon_a \leq J\left(u_t\right)<0$. Together with Lemma \ref{L2}, this implies that $-\infty<\Upsilon_a<0$.
\end{proof}

\begin{lemma} \label{L3} If $0<a_1<a_2$, we have $\frac{a_1^2}{a_2^2}\Upsilon_{a_2}<\Upsilon_{a_1}<0$.

\end{lemma}
\begin{proof} Let $\xi>1$ such that $a_2=\xi a_1$, and let $(u_n) \subset S(a_1)$ be a minimizing sequence with respect to $\Upsilon_{a_1}$, i.e.,
$$
J(u_n) \to \Upsilon_{a_1}\, \,\,\text{as}\,\,
\, n \to +\infty.
$$
Define $v_n={\xi}{u_n}$, we derive that $v_n \in S(a_2)$, and so,
\begin{equation}
    \Upsilon_{a_2} \leq J\left(v_n\right)=\xi^2 J\left(u_n\right)+\frac{\left(\xi^2-\xi^q\right)}{q} \int_{\mathbb{R}^2}\left|u_n\right|^q d x d y .
\end{equation}
We claim that there exist a positive constant $C>0$ and $n_0\in \mathbb{N}$ such that $\displaystyle \int_{\mathbb{R}^2}|u_n|^q\,dxdy \geq C$ for all $n\geq n_0 $.
Otherwise, we have
$$
\int_{\mathbb{R}^2}|u_n|^q\,dxdy \to 0,  \,\,\text{as}\,\,n \to +\infty,
$$
up to a subsequence if necessary. Now, recalling that
$$
0>\Upsilon_{a_1}+o_n(1)=J(u_n)\geq -\frac{1}{q}\int_{\mathbb{R}^2}|u_n|^q\,dxdy, \quad n \in \mathbb{N},
$$
we get a contradiction, and our claim is proved.
Using this claim and the fact that $\xi^2-\xi^q<0$, we obtain that for $n\in \mathbb{N}$ large
$$
\Upsilon_{a_2}\leq \xi^2 J(u_n)+\frac{(\xi^{2}-\xi^q)C}{q}.
$$
Let $n \to +\infty$, one gets
$$
\Upsilon_{a_2}\leq \xi^2 \Upsilon_{a_1}+\frac{(\xi^{2}-\xi^q)C}{q}<\xi^2\Upsilon_{a_1},
$$
that is,
$$
\frac{a_1^2}{a_2^2}\Upsilon_{a_2}<\Upsilon_{a_1},
$$
which proves the lemma.
\end{proof}
\begin{lemma} \label{C2} Let $(u_n)\subset S(a)$ be a minimizing sequence with respect to $\Upsilon_{a}$ such that  $u_n \rightharpoonup u$ in $X$, \linebreak $u_n(x) \to u(x)$ a.e. in $\mathbb{R}^2$ and $u \not=0 $. Then, $u \in S(a)$, $J(u)=\Upsilon_a$ and $u_n \to u$ in $X$.

\end{lemma}
\begin{proof}Indeed, if $|u|_2=b \not=a$, by Fatou's lemma and $u \not=0 $, we must have $b \in (0,a)$. By the Br\'{e}zis-Lieb lemma (see \cite[Lemma 1.32]{Willem}  ),
\begin{equation}\label{BL-1}
|u_n|_2^{2}=|u_n-u|_{2}^{2}+|u|_{2}^{2}+o_n(1)
\end{equation}
and
\begin{equation}\label{BL-2}
|u_n|_q^{q}=|u_n-u|_{q}^{q}+|u|_{q}^{q}+o_n(1).
\end{equation}
Since $u_n \rightharpoonup u$ in $X$, we have
\begin{equation}\label{BL-4}
\|u_n\|_0^{2}=\|u_n-u\|_0^{2}+\|u\|_0^{2}+o_n(1).
\end{equation}
{Let $v_n:=u_n-u$, $d_n:=|v_n|_2$} and suppose that $|v_n|_2 \to d$ as $n \to +\infty$, we deduce that $a^2=b^2+d^2$ and $d_n \in (0,a)$ for $n$ large enough. \eqref{BL-2}-\eqref{BL-4} together with Lemma \ref{L3} imply
\begin{align*}
\Upsilon_{a}+o_n(1)=J(u_n)=&J(v_n)+J(u)+o_n(1)\\
\geq &\Upsilon_{d_n}+\Upsilon_b+o_n(1)\\
\geq &\frac{d_n^2}{a^2}\Upsilon_{a}+\Upsilon_b+o_n(1).
\end{align*}
Let $n \to +\infty$, one finds
    \begin{equation}\label{newine}
    \Upsilon_{a} \geq \frac{d^2}{a^2}\Upsilon_{a}+\Upsilon_b.
    \end{equation}
    Since $b \in (0,a)$, using Lemma \ref{L3} again in \eqref{newine}, we derive the following inequality
    $$
    \Upsilon_{a} > \frac{d^2}{a^2}\Upsilon_{a}+\frac{b^2}{a^2}\Upsilon_{a}=\left(\frac{d^2}{a^2}+\frac{b^2}{a^2}\right)\Upsilon_a=\Upsilon_a,
    $$
    which is absurd. This shows that $|u|_2=a$, that is, $u \in S(a)$.
Since $|u_n|_2=|u|_2=a$ and $u_n \rightharpoonup u$ in $L^{2}(\mathbb{R}^2)$, it follows that
$$
u_n \to u \quad \mbox{in} \quad L^{2}(\mathbb{R}^2).
$$
{The last limit combined} with interpolation theorem in the Lebesgue space gives
$$
u_n \to u \quad \mbox{in} \quad L^{q}(\mathbb{R}^2).
$$
On the other hand, since {$u \mapsto \displaystyle \int_{\mathbb{R}^{2}} \left(|u_{x}|^{2} + |D^{-1}_{x} u_{y}|^{2} \right)dx dy$}
is continuous and convex in $X$, we must have
\begin{align*}
\liminf_{n\to +\infty}\displaystyle\int_{\mathbb{R}^{2}} \left(|(u_{n})_{x}|^{2} + |D^{-1}_{x} (u_{n})_{y}|^{2} \right)dx dy
\geq\displaystyle\int_{\mathbb{R}^{2}} \left(|u_{x}|^{2} + |D^{-1}_{x} u_{y}|^{2} \right)dx dy.
\end{align*} \\
These limits together with $\Upsilon_a=\displaystyle \lim_{n \to +\infty}J(u_n)$ imply that
$$
\Upsilon_a \geq J(u).
$$
Since $u \in S(a)$, we conclude that $J(u)=\Upsilon_a$. Thus, $J(u_n) \to J(u)$.  Moreover, $u_n \to u$ in $L^{q}(\mathbb{R}^2)$ implies that $u_n \to u$ in $X$.
\end{proof}
\noindent\textbf{Proof of Theorem \ref{Th1}:}
{By Lemma \ref{L2},  there exists a bounded minimizing sequence $(u_n) \subset S(a)$ with respect to $\Upsilon_{a}$. We claim that there are $R,\beta>0$ and $(x_n, y_n) \in \mathbb{R}^2$ such that}
\begin{equation} \label{Lions}
\int_{B_R((x_n, y_n))}|u_n|^2\,dx \geq \beta, \,\, \text{for all}\,\, n \in N.
\end{equation}
Otherwise, by Lions-type result for $X$ found in \cite[Lemma 7.4]{Willem}, for any $q \in\left(2, 6\right)$, one has
$$
\int_{\mathbb{R}^{2}}|u_{n}|^{q} dx dy\rightarrow 0, \,\,\text{as}\,\, n \rightarrow +\infty
$$
which is absurd.  {From this, considering $\hat{u}_n(x, y)=u(x+x_n, y+y_n)$, we have that $(\hat{u}_n) \subset S(a)$,  $(\hat{u}_n)$ is also a minimizing sequence with respect to $\Upsilon_{a}$, and we can assume $\hat{u}_n \rightharpoonup \hat{u}$ in $X$ with $\hat{u} \not= 0$ . From Lemma \ref{C2}, $\hat{u} \in S(a)$, $J(\hat{u})=\Upsilon_a$ and $\hat{u}_n \to \hat{u}$ in $X$, finishing the proof of Theorem \ref{Th1}.}

{Next, we prove Theorem \ref{Th-1}, in the remainer of this section, we assume that $f(t)=|t|^{\frac{4}{3}}t$.}

{
\noindent\textbf{Proof of Theorem \ref{Th-1}:}
Let $u\in S(a)$ and set $u_t=e^{t}u(e^{\frac{2}{3}t}x,e^{\frac{4}{3}t}y) $ for $t\in\R$. {A direct computation shows that}
$$
J(u_t)=\frac{e^{\frac{4}{3}t}}{2}\int_{\mathbb{R}^2}\left(|u_x|^2+|D^{-1}_{x} u_{y}|^2\right)\,dxdy-\frac{e^{\frac{4}{3}t}}{q}\int_{\mathbb{R}^2}|u|^{\frac{10}{3}}\,dxdy.
$$
Since $\displaystyle \lim_{t \rightarrow-\infty}J(u_t)=0$, it follows that $\Upsilon_{a}\leq 0$ for any $a>0$. On the other hand,  using \eqref{Gagliardo}, {we obtain for any $u\in S(a)$ that}\label{key}
$$
\begin{aligned}
    J(u) & =\frac{1}{2}\int_{\mathbb{R}^{2}} \left(|u_{x}|^{2} + |D^{-1}_{x} u_{y}|^{2} \right)dx dy-\frac{3}{10}\int_{\mathbb{R}^2} |u|^{\frac{10}{3}} \,dxdy\\
    & \geq \left(\frac{1}{2}-
    \frac{3}{10}C_{\frac{10}{3}}a^{\frac{4}{3}}
    \right)\int_{\mathbb{R}^{2}} \left(|u_{x}|^{2} + |D^{-1}_{x} u_{y}|^{2} \right)dx dy\\
    & \geq \frac{1}{2}\left(1-
    \left(\frac{a}{a^*}\right)^{\frac{4}{3}}
    \right)\int_{\mathbb{R}^{2}} \left(|u_{x}|^{2} + |D^{-1}_{x} u_{y}|^{2} \right)dx dy,
\end{aligned}
$$
where $a^*=\left(\frac{3}{5}C_{\frac{10}{3}}\right)^{-\frac{3}{4}}$. {From this inequality, we deduce that} $\Upsilon_{a}\geq 0$, and hence $\Upsilon_{a}=0$, for $ a \leq a^{*} $.

Next, we prove that there is no nontrivial solution to \eqref{Equation} when $ a \leq a^{*} $.
Indeed, if $ u $ is a solution to \eqref{Equation}, {then by Lemma \ref{Lem-poh} in Section \ref{sec4}, it satisfies the Pohozaev identity}
\begin{equation}
    \frac{2}{3}\int_{\mathbb{R}^{2}}\left( |u_{x}|^{2} +|D^{-1}_{x} u_{y}|^{2}\right) dx dy=\frac{2}{5}\int_{\mathbb{R}^2} |u|^{\frac{10}{3}}dxdy.
\end{equation}
 Applying \eqref{Gagliardo}, we obtain
$$
\frac{2}{3}\int_{\mathbb{R}^{2}}\left( |u_{x}|^{2} +|D^{-1}_{x} u_{y}|^{2}\right) dx dy\leq\frac{2}{5}\left(\frac{a}{a^*}\right)^{\frac{4}{3}}\int_{\mathbb{R}^{2}}\left( |u_{x}|^{2} +|D^{-1}_{x} u_{y}|^{2}\right) dx dy,
$$
{which implies} $ u = 0 $ when $ a \leq a^{*} $.}  This complete the proof.
\section{$L^2$-supercritical case}\label{sec4}
In this section, we assume that {$f(t)=|t|^{q-2}t$} and $q \in (\frac{10}{3},6)$. Define the space $H=X \times \mathbb{R}$, equipped with the inner product
$$
\langle\cdot, \cdot\rangle_{H}=\langle\cdot, \cdot\rangle_{X}+\langle\cdot, \cdot\rangle_{\mathbb{R}},
$$
and the corresponding norm
$$
\Vert \cdot\Vert_{H}=(\Vert \cdot\Vert^{2}+ \vert \cdot\vert_{ \mathbb{R}}^{2})^{1/2}.
$$
{Moreover, we  define the map  $\mathcal{H}: H\rightarrow X$ by }
\begin{equation*}
    \mathcal{H}(u, t)(x, y):=e^{t}u(e^{\frac{2}{3}t}x,e^{\frac{4}{3}t}y).
\end{equation*}
{Direct computations yield}
$$
\int_{\mathbb{R}^2}|\mathcal{H}(u, t)|^2\,dxdy=\int_{\mathbb{R}^2}|u|^2\,dxdy,
$$
$$
\int_{\mathbb{R}^2}|\mathcal{H}(u, t)|^q\,dxdy=e^{(q-2)t}\int_{\mathbb{R}^2}|u|^q\,dxdy,
$$
$$
\int_{\mathbb{R}^2}|(\mathcal{H}(u, t))_x|^2\,dxdy=e^{\frac{4}{3}t}\int_{\mathbb{R}^2}|u_x|^2\,dxdy,
$$
and
$$
\int_{\mathbb{R}^2}|D_x^{-1}(\mathcal{H}(u, t))_y|^2\,dxdy=e^{\frac{4}{3}t}\int_{\mathbb{R}^2}|D^{-1}_{x} u_{y}|^2\,dxdy.
$$
{Using the above notation,} consider the functional $\tilde{J}: X\rightarrow \mathbb{R}$ defined by
$$
{   \tilde{J}(u, t)=J(\mathcal{H}(u, t))}=\frac{e^{\frac{4}{3}t}}{2}\int_{\mathbb{R}^2}\left(|u_x|^2+|D^{-1}_{x} u_{y}|^2\right)\,dxdy-\frac{e^{(q-2)t}}{q}\int_{\mathbb{R}^2}|u|^q\,dxdy
$$
or equivalently,
$$
\tilde{J}(u, t)=\frac{1}{2}\int_{\mathbb{R}^{2}} \left(|v_{x}|^{2} + |D^{-1}_{x} v_{y}|^{2} \right)\,dx dy-\frac{1}{q}\int_{\mathbb{R}^2}|u|^q\,dxdy=J(v),\quad  \mbox{for}\ v=\mathcal{H}(u, t)(x,y).
$$
Adapting some ideas from \cite{Jean}, we are going to prove that $\tilde{J}$ on $S(a) \times \mathbb{R}$ possesses a  mountain-pass geometrical structure.
\begin{lemma}\label{M1}
    Let $u \in S(a)$ be fixed. Then, \\
    (i) $\displaystyle \int_{\mathbb{R}^2}\left| (\mathcal{H}(u, t))_x\right|^2 dxdy+\int_{\mathbb{R}^2}\left|D_x^{-1}(\mathcal{H}(u, t))_y\right|^2 dxdy \rightarrow 0$ and $J(\mathcal{H}(u, t)) \rightarrow 0$ as $t \rightarrow-\infty$;\\
    (ii) $\displaystyle \int_{\mathbb{R}^2}\left| (\mathcal{H}(u, t))_x\right|^2 dxdy+\int_{\mathbb{R}^2}\left|D_x^{-1}(\mathcal{H}(u, t))_y\right|^2 dxdy \rightarrow +\infty$ and $J(\mathcal{H}(u, t)) \rightarrow-\infty$ as $t \rightarrow+\infty$.
\end{lemma}

\begin{proof} \mbox{} By direct calculation,
    \begin{equation} \label{CONV0}\int_{\mathbb{R}^2}\vert \mathcal{H}(u, t)(x, y)\vert^{2}\,dx=a^{2}, \quad \int_{\mathbb{R}^2}|\mathcal{H}(u, s)(x, y)|^q\,dxdy=e^{(q-2)t}\int_{\mathbb{R}^2}|u|^q\,dxdy,
    \end{equation}
    \begin{equation}\label{INE1}
        \int_{\mathbb{R}^2}|\left(\mathcal{H}(u, s)(x, y)\right)_x|^2\,dxdy=e^{\frac{4}{3}t}\int_{\mathbb{R}^2}|u_x|^2\,dxdy
    \end{equation}
    and
    \begin{equation} \label{INE2}
        \int_{\mathbb{R}^2}|D_x^{-1}(\mathcal{H}(u, t))_y|^2\,dxdy=e^{\frac{4}{3}t}\int_{\mathbb{R}^2}|u(x,y)|^2\,dxdy.
    \end{equation}
    By the equalities above, we obtain the limits below
    \begin{equation} \label{CONV1}
        \int_{\mathbb{R}^2}|\left(\mathcal{H}(u, t)(x, y)\right)_x|^2\,dxdy\rightarrow 0 \quad \mbox{and} \quad\int_{\mathbb{R}^2}|D_x^{-1}(\mathcal{H}(u, t))_y|^2\,dxdy \to 0 \quad  \mbox{as} \quad t \to -\infty,
    \end{equation}
    and
    $$
    \int_{\mathbb{R}^2}|\mathcal{H}(u, t)(x, y)|^q\,dxdy=e^{(q-2)t}\int_{\mathbb{R}^2}|u|^q\,dxdy \to 0 \quad \mbox{as} \quad t \to -\infty,
    $$
    which lead to
    $$
    J(\mathcal{H}(u, t))\rightarrow 0 \quad \mbox{as} \quad  t\rightarrow -\infty,
    $$
    showing $(i)$.

    In order to prove $(ii)$, observe that
    $$
    J(\mathcal{H}(u, t))= \frac{e^{\frac{4}{3}t}}{2}\int_{\mathbb{R}^{2}}\left( |u_{x}|^{2}+|D^{-1}_{x} u_{y}|^{2}\right) dx dy-\frac{e^{(q-2)t}}{q}\int_{\mathbb{R}^2} |u|^q \,dxdy.
    $$
    Since $q>\frac{10}{3}$, one has
    $$
    J(\mathcal{H}(u, t))\rightarrow -\infty \quad \mbox{as} \quad t\rightarrow +\infty.
    $$
\end{proof}

\begin{lemma}  \label{PJ1} There exists $K(a)>0$ small enough such that
    $$
    0<\sup_{u\in A} J(u)<\inf_{u\in B} J(u)
    $$
    where
    $$
    A=\left\{u\in S(a), \int_{\mathbb{R}^{2}} \left(|u_{x}|^{2} + |D^{-1}_{x} u_{y}|^{2} \right)dx dy\leq K(a) \right\}
    $$
    and
    $$
    B=\left\{u\in S(a), \int_{\mathbb{R}^{2}} \left(|u_{x}|^{2} + |D^{-1}_{x} u_{y}|^{2} \right)dx dy=2K(a) \right\}.
    $$
\end{lemma}

\begin{proof}
    Using Lemma \ref{L1} again, there exists a constant $C_1>0$,
    depending only on $ a $ and $ q $, such that
    $$
    J(u) \geq \frac{1}{2} \int_{\mathbb{R}^{2}} \left( |u_{x}|^{2} + |D^{-1}_{x} u_{y}|^{2} \right) \, dx \, dy -C_1\left(\int_{\mathbb{R}^2}\left(\left|u_x\right|^2+\left|D_x^{-1} u_y\right|^2\right) d xdy\right)^{\frac{q\beta}{2}}.
    $$
    {For $K(a)>0$ small enough,  this implies that $J(u)>0$ due to $q \beta > 2$.} From this, we must have $\sup _{u \in A} J(u)>0$.

    {Next, fix $u \in A$ and $v \in B$, so $\displaystyle \int_{\mathbb{R}^{2}} \left(|u_{x}|^{2} + |D^{-1}_{x} u_{y}|^{2} \right)dx dy\leq K(a)$ and $\displaystyle \int_{\mathbb{R}^{2}} \left(|v_{x}|^{2} + |D^{-1}_{x} v_{y}|^{2} \right)dx dy=2K(a)$.} By Lemma \ref{L1},
    $$
    \begin{aligned}
        J(v) - J(u) &= \frac{1}{2} \int_{\mathbb{R}^{2}} \left( |v_{x}|^{2} + |D^{-1}_{x} v_{y}|^{2} \right) \, dx \, dy
        - \frac{1}{2} \int_{\mathbb{R}^{2}} \left( |u_{x}|^{2} + |D^{-1}_{x} u_{y}|^{2} \right) \, dx \, dy \\
        &\quad - \frac{1}{q} \int_{\mathbb{R}^2} |v|^q \, dx \, dy
        + \frac{1}{q} \int_{\mathbb{R}^2} |u|^q \, dx \, dy \\
        &\geq \frac{1}{2}K(a)-\frac{1}{q} \int_{\mathbb{R}^2} |v|^q \, dx \, dy.
    \end{aligned}
    $$
    Therefore, there exists a constant $C_2$
    depending only on $ a $ and $ q $, such that
    \[
    J(v) - J(u) \geq \frac{1}{2}K(a) - C_2 \left(K(a)\right)^{\frac{q\beta}{2}},
    \]
    where \( q > \frac{10}{3} \) and \( \beta = \frac{3}{2} - \frac{3}{q} \). Since \( q\beta > 2 \), by choosing \( K(a) > 0 \) sufficiently small such that
    \[
    \frac{1}{2}K(a) - C_2\left(K(a)\right)^{\frac{q\beta}{2}} > 0,
    \]
    {we obtain $\sup_{u\in A} J(u)<\inf_{u\in B} J(u)$. Since $J(u) > 0$ for $u \in A$, the result follows.}
\end{proof}

In what follows, {fix $u_0 \in S(a)$.  by Lemma \ref{M1}, there exist  $s_1<0$ and $s_2>0$ such that  $u_1=\mathcal{H}(s_1,u_0)$ and $u_2=\mathcal{H}(s_2,u_0)$ satisfy}
$$
\| u_1\|^2_0<\frac{K(a)}{2}, \,\, \| u_2\|^2_0>2K(a),\,\, J(u_1)>0\quad \mbox{and} \quad J(u_2)<0.
$$
Following  \cite{Jean}, define the following mountain pass level
$$
\gamma(a)=\inf_{h \in \Gamma}\max_{t \in [0,1]}J(h(t))>0
$$
where
$$
\Gamma=\left\{h \in C([0,1],S(a)): h(0)=u_1 \,\,\mbox{and} \,\, h(1)=u_2 \right\}.
$$
By Lemma \ref{PJ1},
$$
\max_{t \in [0,1]}J(h(t))>\max \left\{J(u_1),J(u_2)\right\}.
$$
{Let $(u_n)$ be a Palais-Smale (PS) sequence for $J|_{S(a)}$ at level $\gamma(a)$, given by $u_n = \mathcal{H}(v_n, s_n)$, where $(v_n, s_n)$ is a (PS) sequence for $\tilde{J}$ as in \cite[Proposition 2.2]{Jean}. Since}
$$
{\partial_s}\tilde{J}(v_n,s_n) \to 0 \quad \mbox{as} \quad n \to +\infty,
$$
it follows that
\begin{equation} \label{EQ1111**}
    P(u_n)\to 0 \quad \mbox{as} \quad n \to +\infty,
\end{equation}
where
\begin{equation} \label{EQ111**}
    P(u)=\frac{2}{3}\int_{\mathbb{R}^{2}}\left( |u_{x}|^{2} +|D^{-1}_{x} u_{y}|^{2}\right) dx dy-\frac{q-2}{q}\int_{\mathbb{R}^2} |u|^q \,dxdy.
\end{equation}
Next, we show that if a $ (PS) $ sequence $(u_n)$ satisfies \eqref{EQ1111**}, then $(u_n)$ is bounded in $X$.
\begin{lemma}\label{Bounded}
    Let $(u_n)$ be a $ (PS) $ sequence for $J|_{S(a)}$ at level $\gamma(a)$, with $P(u_n)\to0$ as $n\rightarrow +\infty$.
    Then $(u_n)$ is bounded in $X$, and up to a subsequence, $u_n \rightharpoonup u$ in $X$. Moreover, there exists $\lambda<0$ such that
    $u $ is a weak solution of the equation
    $$
    \left(-u_{x x}+D_x^{-2} u_{y y}-\lambda u-|u|^{q-2}u\right)_x=0 \,\,\text{in} \,\mathbb{R}^2.
    $$
\end{lemma}
\begin{proof}

    Since $(u_n)$ is a $(PS)$ sequence for $J|_{S(a)}$ at level $\gamma(a)$, we have
    \begin{equation} \label{gamma(a)}
        J(u_n) \to \gamma(a) \quad \mbox{as} \quad n \to +\infty,
    \end{equation}
    and
    {
        \begin{equation*} \label{der11}
            \|J'|_{S(a)}(u_n)\| \to 0 \quad \mbox{as} \quad n \to +\infty.
    \end{equation*}}
    Define the functional $\Psi:X \to \mathbb{R}$ by
    $$
    \Psi(u)=\frac{1}{2}\int_{\mathbb{R}^{2}}u^{2}dxdy,
    $$
    it follows that $S(a)=\Psi^{-1}(\{a^2/2\})$. Then, by Willem \cite[Proposition 5.12]{Willem}, there exists $(\lambda_n) \subset \mathbb{R}$ such that
    \begin{equation}\label{eq-J-psi}
            \|J'(u_n)-\lambda_n\Psi'(u_n)\|_{X^{*}} \to 0 \quad \mbox{as} \quad n \to +\infty.
    \end{equation}
    Hence,
    \begin{equation} \label{EQ10}
        -(u_n)_{xx}+D^{-2}_{x} (u_n)_{yy}-\vert u_n\vert^{q-2}u_n=\lambda_n u_n\ + o_n(1) \quad \mbox{in} \quad X^*.
    \end{equation}
    Moreover, another important limit involving the sequence $(u_n)$ is
    \begin{equation} \label{EQ1**}
        P(u_n)=\frac{2}{3}\int_{\mathbb{R}^{2}}\left( |(u_n)_{x}|^{2} +|D^{-1}_{x} (u_n)_{y}|^{2}\right) dx dy-\frac{q-2}{q}\int_{\mathbb{R}^2} |u_n|^q \,dxdy\to 0 \quad \mbox{as} \quad n \to +\infty,
    \end{equation}
    which is obtained using the limit below
    $$
    {\partial_s}\tilde{J}(v_n,s_n) \to 0 \quad \mbox{as} \quad n \to +\infty.
    $$
    From \eqref{gamma(a)} and  \eqref{EQ1**},  there exists a positive constant $C>0$ independent of $n$ such that
    \begin{equation} \label{BDD1}
        \vert (q-2)J(u_n)-P(u_n)\vert\leq C,
    \end{equation}
    that is
    \begin{equation} \label{BDD2}
        \left(\frac{q-2}{2}-\frac{2}{3}\right)\int_{\mathbb{R}^{2}}\left( |(u_n)_{x}|^{2} +|D^{-1}_{x} (u_n)_{y}|^{2} \right)dx dy\leq C.
    \end{equation}
    Since $q>\frac{10}{3}$, it is easy to see that $\left(\displaystyle\int_{\mathbb{R}^{2}} |(u_n)_{x}|^{2} dx+|D^{-1}_{x} (u_n)_{y}|^{2} dx dy\right)$ is bounded. Moreover, using \eqref{gamma(a)} again, it follows that $\left(\displaystyle \int_{\mathbb{R}^2} |u_n|^q \,dxdy\right)$ is also bounded.
    Recalling that the sequence $\{\lambda_n\}$  must satisfy the equality below
    $$
    \lambda_n=\frac{1}{|u_n|^{2}_{2}}\left\{\int_{\mathbb{R}^{2}}( |(u_n)_{x}|^{2}+|D^{-1}_{x} (u_n)_{y}|^{2}) dx dy-\int_{\mathbb{R}^2} |u_n|^q \,dxdy\right\}+o_n(1)
    $$
    or equivalently,
    \begin{equation} \label{lambdan}
        \lambda_n=\frac{1}{a^{2}}\left\{\int_{\mathbb{R}^{2}}( |(u_n)_{x}|^{2}+|D^{-1}_{x} (u_n)_{y}|^{2}) dx dy-\int_{\mathbb{R}^2} |u_n|^q \,dxdy\right\}+o_n(1),
    \end{equation}
    we also deduce the boundedness of the sequence $(\lambda_n)$. From \eqref{EQ10} and \eqref{EQ1**},
    $$
    (1-\frac{\frac{2}{3}q}{q-2})\left(\int_{\mathbb{R}^{2}} |(u_n)_{x}|^{2} dx dy+\int_{\mathbb{R}^{2}}|D^{-1}_{x} (u_n)_{y}|^{2} dx dy\right)=\lambda_n\int_{\mathbb{R}^2} |u_n|^2 \,dxdy+o_n(1).
    $$
    {Since $q \in (\frac{10}{3},6)$,} we have that $(1-\frac{\frac{2}{3}q}{q-2})<0$, from where it follows that, for some subsequence, still denoted by $(\lambda_n)$,
    \begin{equation}\label{new}
        \lambda_n \to \lambda<0 \quad \mbox{as} \quad n \to +\infty.
    \end{equation}
    Assuming that $u_n\rightharpoonup u$ in $X$, we derive that  $u $ is a weak solution of
    $$
    \left(-u_{x x}+D_x^{-2} u_{y y}-\lambda u-|u|^{q-2}u\right)_x=0, \quad \mbox{in} \quad \mathbb{R}^2.
    $$
\end{proof}

To obtain the compactness of the $ (PS) $ sequence, we shall draw additional variational
characterizations of $\gamma(a)$. Next, we will prove that
$$
\gamma(a)=\inf _{u \in \kappa(a)}J(u)
$$
{where
    $$
    \kappa(a)=\left\{u \in S(a), \,\,J'\mid_{ S(a)}(u)=0\right\}.
    $$}
To this end, we shall present some preliminary results.
\begin{lemma}\label{Lem-poh}
    If $u\in S(a)$ is a weak solution to the equation
    \begin{equation}\label{EQ1}
        \left(-u_{x x}+D_x^{-2} u_{y y}-\lambda u-|u|^{q-2}u\right)_x=0.
    \end{equation}
    then it belongs to the set
    \begin{equation}
        \mathcal{P}(a)=\left\{u \in S(a),P(u)=0\right\},
    \end{equation}
    where
    \begin{equation}
        P(u)=\frac{2}{3}\int_{\mathbb{R}^{2}}\left( |u_{x}|^{2} +|D^{-1}_{x} u_{y}|^{2}\right) dx dy-\frac{q-2}{q}\int_{\mathbb{R}^2} |u|^qdxdy.
    \end{equation}
\end{lemma}
\begin{proof}
    {As proved in \cite[Lemma 2.3]{AGM},}
    if $u$ is a weak solution of equation \eqref{EQ1}, it satisfies the following Pohozaev identity:
    \begin{equation}\label{eq-Q1}
        \frac{1}{2}\int_{\mathbb{R}^{2}}\left( |u_{x}|^{2}+|D^{-1}_{x} u_{y}|^{2} \right)dx dy-\frac{3\lambda}{2}\int_{\mathbb{R}^{2}} u^2dxdy-\frac{3}{q}\int_{\mathbb{R}^{2}} |u|^qdxdy=0 .
    \end{equation}
    On the other hand, since $u$ is a weak solution, it also satisfies
    \begin{equation}\label{eq-J1}
        \int_{\mathbb{R}^{2}}\left( |u_{x}|^{2}+|D^{-1}_{x} u_{y}|^{2} \right)dx dy-\lambda \int_{\mathbb{R}^{2}} u^2dxdy-\int_{\mathbb{R}^{2}} |u|^qdxdy =0.
    \end{equation}
    Combining \eqref{eq-Q1} and \eqref{eq-J1}, we obtain
 $$
    \frac{2}{3}\int_{\mathbb{R}^{2}}\left( |u_{x}|^{2} +|D^{-1}_{x} u_{y}|^{2}\right) dx dy=\frac{q-2}{q}\int_{\mathbb{R}^2} |u|^qdxdy,
    $$
   that is, $P(u)=0$. Hence, $u\in \mathcal{P}(a)$.
\end{proof}	
The proof of the following Lemma \ref{arc} is similar to that of \cite[Lemma 2.8]{Jean}, here we omit it for brevity.
\begin{lemma}\label{arc}
    let $K(a)$ be as defined in Lemma \ref{PJ1}. Then the sets
    $$
    \begin{aligned}
        & A=\left\{u \in S(a),\|u\|_0^2 \leq K(a)\right\} \\
        & {\mathcal{C}=\left\{u \in S(a),\| u\|_0^2 \geq 2 K(a) \text { and } J(u) \leq 0\right\}}
    \end{aligned}
    $$
    are arc-connected. In particular, for any $ v_1 \in A$ and $v_2 \in \mathcal{C}$, we have
    $$
    \gamma(a) = \inf_{g \in \Gamma_{(v_1,v_2)}} \max_{t \in [0,1]} F(g(t))
    $$
    where
    $$
    \Gamma_{(v_1, v_2)} = \{g \in C([0,1], S(a)) : g(0) = v_1, g(1) = v_2\}.
    $$
\end{lemma}
\begin{lemma}\label{Phimax}
    let $u \in S(a)$ be arbitrary but fixed. Then, the function $\Phi_u(t): \R \rightarrow \R$, defined by
    \begin{equation}\label{Phit}
        \Phi_u(t)={J}(\mathcal{H}(u, t))
    \end{equation}
    {attains its unique maximum} at a point $t(u) \in \R$ such that $\mathcal{H}(u, t) \in \mathcal{P}(a)$.
\end{lemma}
\begin{proof}
    Clearly
    $$
    \Phi^{'}_u(t)=\frac{2e^{\frac{4}{3}t}}{3}\int_{\mathbb{R}^{2}}\left( |u_{x}|^{2}+|D^{-1}_{x} u_{y}|^{2}\right) dx dy-\frac{(q-2)e^{(q-2)t}}{q}\int_{\mathbb{R}^2} |u|^q \,dxdy.
    $$
    Since $q>\frac{10}{3}$,  there exists a unique $t_0\in \R$ such that
    $\Phi^{'}_u(t_0)=0$, and $\Phi^{'}_u(t)>0$ for $t\in(-\infty,t_0)$, $\Phi^{'}_u(t)<0$ for $t\in(t_0,+\infty)$. Thus, $t_0$ is the  unique maximum point of $    \Phi_u(t) $.
    By \eqref{Phit}, we note that $ \Phi^{'}_u(t)=P(\mathcal{H}(u, t)) $. Since $ \Phi^{'}_u(t_0)=0 $, it follows that $\mathcal{H}(u, t_0)\in\mathcal{P}(a)$.
\end{proof}

\begin{lemma}
    $\gamma(a)=\inf _{u \in\mathcal{P}(a)} J(u)$.
\end{lemma}
\begin{proof}
    Argue by a contradiction. Suppose that there exists $ v \in \mathcal{P}(a)$ such that
    $J(v)<\gamma(a)$. Define the map $T_v: \R \rightarrow S(a)$ by
    $$
    T_v(t)=\mathcal{H}(v, t).
    $$
    {By Lemma \ref{M1}, there exists $t_0>0$ such that $T_v\left(-t_0\right) \in A$ and {$T_v\left(t_0\right) \in \mathcal{C}$}.} Now, let $\tilde{T}_v$ : $[0,1] \rightarrow S(a)$ be the path defined by
    $$
    \tilde{T}_v(t)=\mathcal{H}\left(v,(2 t-1) t_0\right).
    $$
    Clearly, $\tilde{T}_v(0)=T_v\left(-t_0\right)$ and $\tilde{T}(1)=T_v\left(t_0\right)$. Moreover, by Lemma \ref{Phimax},
    $$
    \gamma(a) \leq \max _{t \in[0,1]} J\left(\tilde{T}_v(t)\right)=J(v),
    $$
    which contradicts the assumption that $ J(v)<\gamma(a) $. Hence, $\gamma(a) = \inf_{u \in \mathcal{P}(a)} J(u) $.
\end{proof}

The proof of Lemma \ref{profile} is similar to \cite[Lemma 5]{WW}, and it will be also omitted for brevity.
\begin{lemma}\label{profile}
    Suppose that $\left(u_n\right) \subset X$ is a bounded (PS) sequence for $J|_{S(a)}(u)$. Then, there exist $\ell \in \mathbb{N}$ and sequences $\left(\widetilde{u}_i\right)_{i=0}^{\ell} \subset X$, $\left((x_n^i,y_n^i)\right)_{i=0}^{\ell} \subset \mathbb{R}^2$ for any $n \geq 1$, with $(x_n^0,y_n^0)=(0,0)$, such that
    $$
    (x_n^i-x_n^j)^2+(y_n^i-y_n^j)^2 \rightarrow +\infty \quad \text{as} \quad n \rightarrow +\infty \quad \text{for} \quad i \neq j,
    $$
    and, passing to a subsequence, the following hold for any $i \geq 0$:
    \begin{equation}\label{eq:com22}
        \begin{aligned}
            u_n\left(\cdot-x_n^i,\cdot-y_n^i\right) \rightharpoonup& \widetilde{u}_i \text{ in } X \text{ as } n \rightarrow + \infty\quad \text{and} {\quad J'|_{S(a_j)}(\widetilde{u}_j)=0, \quad \mbox{where} \quad a_j=|\widetilde{u}_j|_2 \quad \mbox{and} \quad a^2=\sum_{i=1}^{\ell}b_i^2. }
        \end{aligned}
    \end{equation}
    \begin{equation}\label{eq:com1}
        \begin{aligned}
            \left\|u_n-\sum_{i=0}^{\ell} \widetilde{u}_i\left(\cdot+x_n^i, \cdot+y_n^i\right)\right\| \rightarrow 0 \quad \text { as } n \rightarrow +\infty,
        \end{aligned}
    \end{equation}
    \begin{equation}\label{eq:com3}
        \begin{aligned}
            \sum_{i=0}^{\ell} J\left(\widetilde{u}_i\right)={\lim_{n \rightarrow+ \infty}J({u}_n)}.
        \end{aligned}
    \end{equation}
\end{lemma}
\begin{lemma}\label{subadd}
    For $q\in (\frac{10}{3}, 6)$, let $k \in \N$ and $a, a_1, a_2, \ldots, a_k>0$ satisfy $a^2=a_1^2+\ldots+a_k^2$. Then,
    $$
    \gamma(a)<\gamma\left(a_1\right)+\ldots .+\gamma\left(a_k\right).
    $$
\end{lemma}
\begin{proof}
    We first  prove that
    $$
    \gamma(\theta a)<\theta^2 \gamma(a), \quad \forall a>0 \quad \mbox{and} \quad \theta>1.
    $$
    Let $u_n=\mathcal{H}(v_n,s_n)$,  where $(v_n,s_n)$ is a $(PS)$ sequence for $ \tilde{J} $ as established in \cite[Proposition 2.2]{Jean}, then
    \begin{equation}\label{eq:Q}
        J(u_n)\rightarrow \gamma(a)\quad \text{and}\quad P(u_n)\rightarrow 0 \quad as\quad n\rightarrow +\infty.
    \end{equation}
    By Lemma \ref{Phimax}, for any $u\in X$, {there exists a unique $t\in \mathbb{R} $ such that}
    $$
    {\partial_t}\tilde{J}(u,t)=P(\mathcal{H}(u, t))=0.
    $$
    Therefore, for all $n \in \mathbb{N}$, we define $t(n, \theta) \in \mathbb{R}$ as the unique value satisfying
    $$
    P(\mathcal{H}\left(\theta u_n, t(n, \theta)\right) )=0.
    $$
    In addition, from \eqref{eq:Q},  we have $t(n, \theta)\to 0$ as $n \rightarrow +\infty$.
    Then
    \begin{equation}
        \begin{aligned}
            \gamma(\theta a)  \leq& J\left(\mathcal{H}\left(\theta u_n, t(n, \theta)\right)\right) \\
            =&\frac{1}{2}\int_{\mathbb{R}^{2}} |(\mathcal{H}\left(\theta u_n, t(n, \theta)\right))_{x}|^{2} dx dy+\frac{1}{2}\int_{\mathbb{R}^{2}}|D^{-1}_{x} (\mathcal{H}\left(\theta u_n, t(n, \theta)\right))_{y}|^{2} dx dy\\
            &-\frac{1}{q}\int_{\mathbb{R}^2} |\mathcal{H}\left(\theta u_n, t(n, \theta)\right)|^qdxdy \\
            =&\frac{\theta^2}{2}\left(\int_{\mathbb{R}^{2}} |(\mathcal{H}\left( u_n, t(n, \theta)\right))_{x}|^{2} dx dy+\int_{\mathbb{R}^{2}}|D^{-1}_{x} (H\left( u_n, t(n, \theta)\right))_{y}|^{2} dx dy\right)\\
            &-\frac{\theta^q}{q}\int_{\mathbb{R}^2} |\mathcal{H}\left( u_n, t(n, \theta)\right)|^qdxdy \\
            <&\frac{\theta^2}{2}\left(\int_{\mathbb{R}^{2}} |(\mathcal{H}\left( u_n, t(n, \theta)\right))_{x}|^{2} dx dy+\int_{\mathbb{R}^{2}}|D^{-1}_{x} (H\left( u_n, t(n, \theta)\right))_{y}|^{2} dx dy\right)\\
            &-\frac{\theta^2}{q}\int_{\mathbb{R}^2} |\mathcal{H}\left( u_n, t(n, \theta)\right)|^q dxdy \\
            =&\theta^2 J\left(\mathcal{H}\left( u_n, t(n, \theta)\right)\right) \\
            =& \theta^2 J\left(\mathcal{H}\left( u_n, 0\right)\right)+\theta^2\left(J\left(\mathcal{H}\left( u_n, t(n, \theta)\right)\right)-J\left(\mathcal{H}\left( u_n, 0\right)\right)\right)  \\
            =&\theta^2 J\left(u_n\right)+ o_n(1).
        \end{aligned}
    \end{equation}
    Tke the limit as $n \rightarrow+\infty$, we obtain that $\gamma(\theta a) \leq \theta^2 \gamma(a)$.  Equality holds only if
    \begin{equation}\label{eq-H1}
        \int_{\mathbb{R}^2} |\mathcal{H}\left(\theta u_n, t(n, \theta)\right)|^qdxdy \to 0, \quad \text{as} \quad n \rightarrow+\infty.
    \end{equation}
    Since $P(\mathcal{H}\left(\theta u_n, s(n, \theta)\right) )=0$, for all $n \in \N$, it follows that, as $ n \rightarrow+\infty$,
    \begin{equation}\label{eq-H2}
        \int_{\mathbb{R}^{2}} |(\mathcal{H}\left(\theta u_n, t(n, \theta)\right))_{x}|^{2} dx dy+\int_{\mathbb{R}^{2}}|D^{-1}_{x} (\mathcal{H}\left(\theta u_n, t(n, \theta)\right))_{y}|^{2} dx dy \to 0.
    \end{equation}
    {From the definition of $J$}, combining \eqref{eq-H1} and \eqref{eq-H2}, we deduce
    $$
    J\left(\mathcal{H}\left(\theta u_n, t(n, \theta)\right)\right) \rightarrow 0 \quad \text{as} \quad n \rightarrow+\infty.
    $$
    This contradicts the fact that $\gamma(a)>0$ for all $a>0$. Thus, the strict inequality holds.

    {For completeness,  we recall the proof for the general case}. Suppose first that $k=2$ and $a_1 \geq a_2$. Then,
    $$
    \begin{aligned}
        \gamma(a) & <\frac{a^2}{a_1^2} \gamma\left(a_1\right) \\
        & =\gamma\left(a_1\right)+\frac{a_2^2}{a_1^2} \gamma\left(a_1\right) \\
        & <\gamma\left(a_1\right)+\gamma\left(a_2\right) .
    \end{aligned}
    $$
    For $k>2$, assume $a_1 \geq \ldots \geq a_k$ and that the assertion holds for $k-1$. Setting  $\tilde{a}=$ $\sqrt{a_1^2+\ldots+a_{k-1}^2}$, we have
    $$
    \begin{aligned}
        \gamma(a) & <\frac{a^2}{\tilde{a}^2} \gamma(\tilde{a}) \\
        & =\gamma(\tilde{a})+\frac{a_k^2}{\tilde{a}^2} \gamma(\tilde{a})\\
        & <\gamma(\tilde{a})+\gamma\left(a_k\right) \\
        & <\gamma\left(a_1\right)+\ldots+\gamma\left(a_k\right).
    \end{aligned}
    $$
    This completes the proof of the lemma.
\end{proof}
\noindent\textbf{Proof of Theorem \ref{Th2}:}
We turn back to our $(PS)$ sequence
$u_n=\mathcal{H}(v_n,s_n)$,  where $(v_n,s_n)$ is the $(PS)$ sequence for $ \tilde{J} $ obtained from \cite[Proposition 2.2]{Jean}.
{Then, from Lemma \ref{Bounded}, $(u_n)$ is bounded in $X$.} Thus, by Lemma \ref{profile},
there exist $ (\widetilde{u}_i) $
such that,
$$
a^2=\lim _{n \rightarrow+ \infty} \left| u_n\right|^2_2 =\sum_{i=0}^{\ell} \left| \widetilde{u}_i\right|^2_2 \quad \text{ and } \quad \lim _{n \rightarrow+\infty} J\left(u_n\right)=\sum_{i=0}^{\ell} J\left(\tilde{u}_j\right).
$$
We aim to prove that $\ell=0$. Set $a_i=|\widetilde{u}_i|_2$ . Suppose, for contradiction, that $\ell \geq 1$. Then, from the variational characterization $\gamma(a)=\inf_{u \in \kappa(a)} J(u)$ and Lemma \ref{subadd}, we  have
$$
\gamma(a)=\sum_{i=0}^\ell J\left(\tilde{u}_i\right) \geq \sum_{i=0}^{\ell} \gamma\left({a}_i\right)>\gamma(a).
$$
This is a contradiction. Hence, $i=0$, and $u_n\to u$ in $X$ as $n\to+\infty$. Thereby, there exists ${u} \in S(a)$ such that
$J(u)=\gamma(a)$, and the proof of Theorem \ref{Th2} is complete.

  \section{{Combined power nonlinearities}}
  In this section, we assume that $f(t)=\mu|t|^{q-2}t+|t|^{p-2}t$, with {$2<q<\frac{10}{3}<p<6$}. {The aim of the section is to investigate the existence of  solutions with negative energy for \eqref{Equation}. Furthermore,  for a sequence $(a_n) \subset (0,a_0)$ with $a_n \to 0$ as $n \to+\infty$, we will show that problem \eqref{Equation} with $a=a_n$ admits a second solution with positive energy.}
  {For clarity and to distinguish it from the previous sections, we use \( J_\mu \) to denote the original energy functional \( J \) in this section.}
  We begin by establishing an appropriate estimate for the energy functional $J_\mu: X \rightarrow \R$ defined on ${S}(a)$ by
   \begin{equation*}
    J_{\mu}(u)=\frac{1}{2}\int_{\mathbb{R}^{2}} \left(|u_{x}|^{2} + |D^{-1}_{x} u_{y}|^{2} \right)dx dy-\frac{\mu}{q}\int_{\mathbb{R}^2} |u|^q \,dxdy-\frac{1}{p}\int_{\mathbb{R}^2} |u|^p \,dxdy.
  \end{equation*}
   }

  First of all, note that
  \begin{equation}\label{eqJF}
    \begin{aligned}
        J_\mu(u) & =\frac{1}{2}\|u\|_0^2-\frac{\mu}{q}|u|_q^q-\frac{1}{p}|u|_p^p \\
        & \geq \frac{1}{2}\|u\|_0^2-\frac{\mu}{q}C_q\|u\|_0^{q \beta_q}|u|_2^{(1-\beta_q) q}-\frac{1}{p}C_p\|u\|_0^{p \beta_p}|u|_2^{(1-\beta_p) p}.
    \end{aligned}
  \end{equation}
  Define the function {$h:(0,\infty) \times(0,\infty) \rightarrow \mathbb{R}$} by
  \begin{equation}
  {	h(a,\rho)=\frac{1}{2}-\frac{\mu}{q}C_q\rho^{q\beta_q-2}a^{(1-\beta_q) q}-\frac{1}{p}C_p\rho^{p\beta_p-2}a^{(1-\beta_p) p}.}
  \end{equation}
  and, for each $a \in(0, \infty)$, define its restriction $g_a(\rho)$ on $(0, \infty)$ by $\rho \mapsto {g_a(\rho):=h(a, \rho)}$.
  Then
  \begin{equation}\label{Eq-Phi>f}
    J_{\mu}(u) \geq {h\left(a,\|u\|_{0}\right)\| u\|_{0}^2}, \quad \text { for all } u \in \mathcal{D}(a).
  \end{equation}
  \begin{lemma}\label{lem5-g}
  For each $a>0$, the function $g_a(\rho)$ has a unique global maximum, and the maximum value satisfies
  $$
  \left\{\begin{array}{lll}
    \max _{\rho>0} g_a(\rho)>0 & \text { if } & a<a_0 \\
    \max _{\rho>0} g_a(\rho)=0 & \text { if } & a=a_0 \\
    \max _{\rho>0} g_a  (\rho)<0 & \text { if } & a>a_0
  \end{array}\right.
  $$
  where
  \begin{equation}\label{eq-a0}
  a_0:=\left(\frac{1}{2 K}\right)^{\frac{1}{2}}>0
  \end{equation}
  with
  $$
  K:=\frac{\mu}{q}C_q
  \left(-\frac{\left(q \beta_q-2\right)}{\left(p \beta_p-2\right)}
  \frac{p \mu}{q}
  \frac{C_q}{C_p}\right)
  ^{\frac{q\beta_q-2}{p \beta_p-q \beta_q}}
  +
  \frac{1}{p}C_p
  \left(-\frac{\left(q \beta_q-2\right)}{\left(p \beta_p-2\right)}
  \frac{p \mu}{q}
  \frac{C_q}{C_p}\right)
  ^{\frac{p\beta_p-2}{p \beta_p-q \beta_q}}.
  $$
  \end{lemma}
    \begin{proof}
    By the definition of $g_c(\rho)$, we have
    \begin{equation}
        g_a^{\prime}(\rho)=-(q\beta_q-2)\frac{\mu}{q}C_q\rho^{q\beta_q-3}a^{(1-\beta_q) q}-(p\beta_p-2)\frac{1}{p}C_p\rho^{p\beta_p-3}a^{(1-\beta_p) p} .
    \end{equation}
    Hence, the equation $g_a^{\prime}(\rho)=0$ has a unique solution given by
    \begin{equation}\label{eq-rho0}
    \rho_a= \left(-\frac{\left(q \beta_q-2\right)}{\left(p \beta_p-2\right)}
     \frac{p \mu}{q}
     \frac{C_q}{C_p}\right)
     ^{\frac{1}{p \beta_p-q \beta_q}}
     a^{\frac{\left(1-\beta_q\right) q-\left(1-\beta_p\right)p}{p \beta_p-q \beta_q}}.
    \end{equation}
    Taking into account that $g_a(\rho) \rightarrow-\infty$ as $\rho \rightarrow 0$ and $g_a(\rho) \rightarrow-\infty$ as $\rho \rightarrow \infty$, we deduce that $\rho_a$ is the unique global maximum point of $g_a(\rho)$. The corresponding maximum value is
    \begin{equation}
    \begin{aligned}
    \max _{\rho>0} g_a(\rho)  =&
    \frac{1}{2}-
    \frac{\mu}{q}C_q
    \left(-\frac{\left(q \beta_q-2\right)}{\left(p \beta_p-2\right)}
    \frac{p \mu}{q}
    \frac{C_q}{C_p}\right)
    ^{\frac{q\beta_q-2}{p \beta_p-q \beta_q}}
    a^{{\frac{\left(1-\beta_q\right) q-\left(1-\beta_p\right)p}{p \beta_p-q \beta_q}}
    ({q\beta_q-2})+{(1-\beta_q) q}}\\
    &-\frac{1}{p}C_p
    \left(-\frac{\left(q \beta_q-2\right)}{\left(p \beta_p-2\right)}
    \frac{p \mu}{q}
    \frac{C_q}{C_p}\right)
    ^{\frac{p\beta_p-2}{p \beta_p-q \beta_q}}
    a^{{\frac{\left(1-\beta_q\right) q-\left(1-\beta_p\right)p}{p \beta_p-q \beta_q}}
        ({p\beta_p-2})+{(1-\beta_q) q}}\\
    =&\frac{1}{2}-
    \left(\frac{\mu}{q}C_q
    \left(-\frac{\left(q \beta_q-2\right)}{\left(p \beta_p-2\right)}
    \frac{p \mu}{q}
    \frac{C_q}{C_p}\right)
    ^{\frac{q\beta_q-2}{p \beta_p-q \beta_q}}
    +
    \frac{1}{p}C_p
    \left(-\frac{\left(q \beta_q-2\right)}{\left(p \beta_p-2\right)}
    \frac{p \mu}{q}
    \frac{C_q}{C_p}\right)
    ^{\frac{p\beta_p-2}{p \beta_p-q \beta_q}}
    \right)
    a^{\frac{pq(\beta_p-\beta_q)}{p\beta_p-q\beta_q}}\\
    &=\frac{1}{2}-Ka^2.
    \end{aligned}
    \end{equation}
    By the definition of $a_0$, we have that $\max _{\rho>0} g_{a_0}(\rho)=0$, and the proof is complete.
    \end{proof}
    \begin{lemma}\label{lem5-f}
    Let $\left(a_1, \rho_1\right) \in(0, \infty) \times(0, \infty)$ be such that $h\left(a_1, \rho_1\right) \geq 0$. Then, for any $a_2 \in\left(0, a_1\right]$, we have
    $$
    h\left(a_2, \rho_2\right) \geq 0 \quad \text {for all} \quad \rho_2 \in\left[\frac{a_2}{a_1} \rho_1, \rho_1\right] .
    $$
    \end{lemma}
    \begin{proof}
    Since $a \rightarrow h(\cdot, \rho)$ is non-increasing,  it follows immediately that
    \begin{equation}\label{eq-f1}
    h\left(a_2, \rho_1\right) \geq h\left(a_1, \rho_1\right) \geq 0.
    \end{equation}
    Now, noting that $\alpha_0+\alpha_1=q-2>0$, a direct calculation yields
    \begin{equation}\label{eq-f2}
    h\left(a_2, \frac{a_2}{a_1} \rho_1\right) \geq f\left(a_1, \rho_1\right) \geq 0.
    \end{equation}
    Now observe that if $g_{a_2}\left(\rho^{\prime}\right) \geq 0$ and $g_{a_2}\left(\rho^{\prime \prime}\right) \geq 0$, then
    \begin{equation}\label{eq-f3}
    h\left(a_2, \rho\right)=g_{a_2}(\rho) \geq 0 \quad \text { for any } \quad \rho \in\left[\rho^{\prime}, \rho^{\prime \prime}\right].
    \end{equation}
    {Indeed,  if there exists some  $\rho \in\left(\rho^{\prime}, \rho^{\prime \prime}\right)$ such that $g_{a_2}(\rho)<0$,} then there exists a local minimum point on $\left(\rho_1, \rho_2\right)$,  contradicting the fact that the function $g_{a_2}(\rho)$ has a unique critical point which is necessarily its unique global maximum (see Lemma \ref{lem5-g}). By \eqref{eq-f1} and  \eqref{eq-f2}, we can choose $\rho^{\prime}=\left(a_2 / a_1\right) \rho_1$ and $\rho^{\prime \prime}=\rho_1$, and \eqref{eq-f3} implies the lemma.
    \end{proof}
    Now let $a_0>0$ be defined by \eqref{eq-a0} and $\rho_0:=\rho_{a_0}>0$ being determined by \eqref{eq-rho0}. Note that by {Lemmas \ref{lem5-g} and \ref{lem5-f}}, we have that $h\left(a_0, \rho_0\right)=0$ and {$h\left(a, \rho_0\right)>0$ for all $a \in\left(0,a_0\right)$}. {In what follows, let us fix the sets below
    $$
    B_{\rho_0}:=\left\{u \in X:\|u\|_0^2<\rho_0\right\},\quad V(a):=S(a) \cap B_{\rho_0}
    $$
    and
    $$
    \partial V(a):=\{u \in S(a):\| u\|_{0}=\rho_0\}.
    $$
    Using the above notations, we are able to consider the following local minimization problem:
    $$
    m(a):=\inf _{u \in V(a)} J_\mu(u), \quad \mbox{for} \quad a \in\left(0, a_0\right).
    $$}
\begin{lemma}\label{Lem5-m(a)}
    For any $a \in\left(0, a_0\right)$,
    \begin{equation}\label{eq-m<0}
    m(a)=\inf _{u \in V(a)} J_\mu(u)<0<\inf _{u \in \partial V(a)} J_\mu(u).
    \end{equation}
\end{lemma}	
\begin{proof}
    For any $u \in \partial V(a)$,  we have $\| u\|_{0}=\rho_0$. {Thereby, by \eqref{eqJF}, }
    $$
    0<\inf _{u \in \partial V(a)} J_\mu(u).
    $$
    Now let $u \in {S}(a)$ be arbitrary but fixed. For $t \in(-\infty, \infty)$, recall that
    $$
    u_t(x,y)=\mathcal{H}(u,t)=e^{t}u(e^{\frac{2}{3}t}x,e^{\frac{4}{3}t}y).
    $$
    Clearly $u_t \in {S}(a)$ for any $t \in(-\infty, \infty)$. {Next, let us define the map $\psi_u: \mathbb{R} \to \mathbb{R}$ by}
    $$
    \psi_u(t):=J_\mu\left(u_t\right)=
    \frac{1}{2}e^{\frac{4}{3}t}\|u\|_0^2-\frac{\mu}{q}{e^{(q-2)t}}|u|^q_q-\frac{1}{p}{e^{(p-2)t}}|u|^p_p.
    $$
    Since $2<q<\frac{10}{3}<p<6$, we have { $\psi_u(t) \to 0^{-}$ as $t \to -\infty$}. Thus, there exists $t_0<0$ such that
    \[
    \left\|u_{t_0}\right\|_{0}^2 = e^{\frac{4}{3}t_0} \| u\|_{0}^2 < \rho_0^2 \quad \mbox{and} \quad J_\mu(u_{t_0}) = \psi_\mu(t_0) < 0.
    \]
    {Hence $u_{t_0} \in V(a)$ and $m(a) < 0$, completing the proof.}
\end{proof}
\begin{lemma}\label{Lem4-ground}
    {For any $a \in\left(0, a_0\right)$, if $m(a)$ is attained, then any {normalized ground state solution} lies in $V(a)$.}
\end{lemma}
\begin{proof}
    It is well known that all critical points of $J_\mu$ restricted to $S(a)$ belong to the Pohozaev's type set
    $$
    \mathcal{P}_{\mu,a}:=\{u \in S(a): {P}_{\mu}(u)=0\}
    $$
    where
    $$
    {P}_{\mu}(u):=\|u\|_0^2-\frac{\mu (q-2)}{ q}|u|_q^q-\frac{(p-2)}{p}|u|_p^p.
    $$
    A direct calculation shows that, for any $v \in S(a)$ and any $t \in(-\infty, \infty)$,
    \begin{equation}\label{Eq-psi_v}
        \psi_v^{\prime}(t)={P}_{\mu}\left(v_t\right)=
        \frac{2}{3}e^{\frac{4}{3}t}\|u\|_0^2-
        \frac{\mu (q-2)}{q}e^{(q-2)t}|u|_q^q-\frac{(p-2)}{p}e^{(p-2)t}|u|_p^p,
    \end{equation}
where $\psi_v^{\prime}$ denotes the derivative of $\psi_v$ with respect to $t \in(-\infty, \infty)$ and $ v_t(x):=e^{t}v(e^{\frac{2}{3}t}x,e^{\frac{4}{3}t}y) $. Finally, observe that any $u \in S(a)$ can be written as {$u=v_t$ with $v \in S(a)$, $\|v\|_{0}=1$ and $t \in(-\infty, \infty)$.}

Since the set $\mathcal{P}_{\mu,a}$ contains all the normalized ground state solutions (if any), we deduce from \eqref{Eq-psi_v} that if $w \in S(a)$ is a normalized ground state solution, then
    $
    {P}_{\mu}\left(w\right)=0.
    $
Thus, there exists a $v \in S(a)$, $\|v\|_{0}=1$ and a $t_0 \in(-\infty, \infty)$ such that $w=v_{t_0}$, $J_\mu(w)=\psi_v\left(t_0\right)$ and $\psi_v^{\prime}\left(t_0\right)={P}_{\mu}\left(w\right)=0$. Namely, $t_0 \in(-\infty, \infty)$ is a critical point of $\psi_v$.

    Now, since $\psi_v(t) \rightarrow 0^{-}$ and $\left\| v_t\right\|_{0} \rightarrow 0$, as $t \rightarrow 0$, and $\psi_v(t)=J_\mu\left(v_t\right) \geq 0$ when
     $$
     v_t \in \partial V(a)=\{u \in\left.S(a):\| u\|_{0}=\rho_0\right\},
     $$
the function $\psi_v$ must has a first critical point $t_1 \in \mathbb{R}$ where $\psi_v'(t_1) = 0$, corresponding to a local minimum.  In particular, $v_{t_1} \in V(a)$ and $J_\mu\left(v_{t_1}\right)=\psi_v\left(t_1\right)<0$. Also, from $\psi_v\left(t_1\right)<0, \psi_v(t) \geq 0$ when $v_t \in \partial V(a)$ and $\psi_v(t) \rightarrow-\infty$ as $t \rightarrow +\infty$, {the function $\psi_v$ has a second critical point $t_2 > t_1$ where $\psi_v'(t_2) = 0$, corresponding to a local maximum.} Since $v_{t_2}$ satisfies $J_\mu\left(v_{t_2}\right)=\psi_v\left(t_2\right) \geq 0$, we have that $m(a) \leq J_\mu\left(v_{t_1}\right)<J_\mu\left(v_{t_2}\right)$. {Thus, since $m(a)$ is attained, $v_{t_2}$ cannot be a normalized ground state solution. Hence, $t_0=t_1$ and $w=v_{t_1} \in V(a)$, completing the proof. }
     \end{proof}
{Our next goal is to establish several technical lemmas to prove the compactness of the minimizing sequences.}
\begin{lemma}\label{Lem4-subadd}
    {For any $a \in\left(0, a_0\right)$ and $b \in(0, a)$,  we have $m(a) \leq m(b)+m(\sqrt{ a^2-b^2})$
with  strict inequality if either $m(b)$ or $m(\sqrt{ a^2-b^2})$ is attained.}
\end{lemma}
\begin{proof}
    {Note that, fixed $b \in(0, a)$, it is sufficient to show that
    \begin{equation}\label{Eq-mtha}
        m(\theta b) \leq \theta^2 m(b),\,\, \forall \theta \in\left(1, \frac{a}{b}\right],
    \end{equation}
    with strict inequality  if $m(b)$ is attained.} Indeed, if \eqref{Eq-mtha} holds, one has
    \begin{equation}
        \begin{aligned}
            m(a) =\frac{a^2-b^2}{a} m(a)+\frac{b^2}{a^2} m(a)&=\frac{a^2-b^2}{a^2} m\left(\frac{a}{\sqrt{a^2-b^2}}(\sqrt{a^2-b^2})\right)+\frac{b^2}{a^2} m\left(\frac{a}{b} b\right) \\
            & \leq m(\sqrt{a^2-b^2})+m(b)
        \end{aligned}
    \end{equation}
    with strict inequality if $m(b)$ is attained.

{Now, for fixed $b \in(0, a)$, we prove \eqref{Eq-mtha}. By Lemma \ref{Lem5-m(a)}, for any $\varepsilon>0$ sufficiently small, there exists  $u \in V(b)$ such that}
    \begin{equation}\label{Eq-Phiu<0}
        J_\mu(u) \leq m(b)+\varepsilon \quad \text { and } \quad J_\mu(u)<0.
    \end{equation}
    By Lemma \ref{lem5-f}, $h(b, \rho) \geq 0$ for any $\rho \in\left[ \frac{b}{a} \rho_0, \rho_0\right]$. Hence, we can deduce from Lemma \ref{Lem5-m(a)} and \eqref{Eq-Phiu<0} that
    $$
    \| u\|_{0}<\frac{b}{a} \rho_0.
    $$
    Define $v(x, y):=u\left(x / \theta^\frac{2}{3}, y / \theta^\frac{4}{3}\right) $. Since  $|v|_2=\theta|u|_2=\theta b$, we have
{   $$
    \| v\|_{0}^{2}=\theta^{2/3}\| u\|_{0}^{2}<\theta^{2/3}\left(\frac{b}{a} \right)^{2}\rho_0^{2}\leq\rho_0^{2},
    $$}
    so $v \in V(\theta b)$. We can write

{   $$
    \begin{aligned}
        m(\theta b) & \leq J_\mu(v)\\
        &=\frac{1}{2} \theta^{2/3}\|u\|_0^2-\frac{\mu}{q} \theta^{2}|u|_q^q-\frac{1}{p} \theta^{2}|u|_{p}^{p} \\
        & <\frac{1}{2} \theta^{2}\|u\|_0^2-\frac{\mu}{q} \theta^{2}|u|_q^q-\frac{1}{p} \theta^{2}|u|_{p}^{p}\\
        &=\theta^{2} J_\mu(u)\\
        & \leq \theta^2(m(b)+\varepsilon)
    \end{aligned}
    $$}
    Since $\varepsilon>0$ is arbitrary, we obtain that $m(\theta b) \leq \theta m(b)$. {If $m(b)$ is attained,  we can set $\varepsilon=0$ in \eqref{Eq-Phiu<0}, yielding strict inequality. This completes the proof.}
    \end{proof}
    \begin{lemma}\label{Lem4-m(a)cont}
    {   The map $a \in (0, a_0) \mapsto m(a)$ is continuous.}
    \end{lemma}
    \begin{proof}
        Let $a \in\left(0, a_0\right)$ be arbitrary and $\left(a_n\right) \subset\left(0, a_0\right)$ with $a_n \rightarrow a$ as $n\rightarrow\infty$.
        It suffices to show that
        $m\left(a_n\right) \rightarrow m(a)$.
        {From the definition of $m\left(a_n\right)$, $m\left(a_n\right)<0$ and Lemma \ref{Lem5-m(a)},} for any $\varepsilon>0$ sufficiently small, there exists $u_n \in V\left(a_n\right)$ such that
        \begin{equation}\label{Eq-Phiun<0}
            J_\mu\left(u_n\right) \leq m\left(a_n\right)+\varepsilon \quad \text { and } \quad J_\mu\left(u_n\right)<0.
        \end{equation}
        {We first show that $(u_n)$ is bounded in $X$. Since $\left\| u_n\right\|_{0}<\rho_0$, the sequences {$(\left|u_n\right|_p)$ and $(\left|u_n\right|_{q})$  are bounded  in $\mathbb{R}$,} and from \eqref{Eq-Phiun<0},  $(u_n)$ is also bounded in $X$. Define $v_n:={\frac{a}{a_n}} u_n$, so $v_n \in {S}(a)$. Next, we are going to prove that}
        \begin{equation}\label{Eq-mPhiv}
            m(a) \leq J_\mu\left(v_n\right).
        \end{equation}
        First, we consider case that $a_n \geq a$, then
        \begin{equation}
            \left\| v_n\right\|_{0}=\frac{a}{a_n}\left\| u_n\right\|_{0} \leq\left\| u_n\right\|_{0}<\rho_0.
        \end{equation}
        Since we have that $v_n \in V(a)$, \eqref{Eq-mPhiv} holds.

        Second, we consider case that $a_n \leq a$, by Lemma \ref{lem5-f},
        $ h\left(a, \rho_0\right) > 0 $.
        Due to the continuity of $\rho \rightarrow h(a, \cdot)$, we may assume there exists sufficiently small
        $\delta>0$ such that
        $$
        h\left(a, \rho\right) > 0\quad \text{for}\quad \rho \in\left[\rho_0, (1+\delta)\rho_0\right].
        $$
        Hence, we deduce from \eqref{Eq-Phi>f} and \eqref{Eq-Phiun<0} that
        $\left\| u_n\right\|_0<\frac{a_n}{a}  \rho_0$ and
        for sufficiently large $n$
        $$
        \left\| v_n\right\|_{0}=\frac{a}{a_n}\left\| u_n\right\|_{0}< \rho_0\leq(1+\delta)\rho_0.
        $$
        If $\rho_0\leq\left\|v_n\right\|_{0}\leq(1+\delta)\rho_0$, from
        \eqref{Eq-Phi>f}
        $$
        m(a_n)<0<h\left(c,\|v_n\|_{0}\right)\| v_n\|_{0}^{2}\leq J_\mu\left(v_n\right).
        $$
        If $\left\|v_n\right\|_{0}<\rho_0$, we have that $v_n \in V(a)$, so \eqref{Eq-mPhiv} holds. Based on the above two cases, we can conclude that
        \begin{equation}
            m(a) \leq J_\mu\left(v_n\right)=J_\mu\left(u_n\right)+\left(J_\mu\left(v_n\right)-J_\mu\left(u_n\right)\right)
        \end{equation}
        and
        \begin{equation}
            \begin{aligned}
                J_\mu\left(v_n\right) - J_\mu\left(u_n\right) = & -\frac{1}{2} \left( \frac{a}{a_n} - 1 \right) \left\| u_n\right\|_0^2
                -
                 + \frac{\mu}{q} \left[ \left( \frac{a}{a_n} \right)^{\frac{q}{2}} - 1 \right] \left|u_n\right|_q^q \\
                &+ \frac{1}{p} \left[ \left( \frac{a}{a_n} \right)^{\frac{p}{2}} - 1 \right] \left|u_n\right|_{p}^{p}.
            \end{aligned}
        \end{equation}
        Since $ (u_n) $ is bounded in $X$, {we infer that}
        \begin{equation}\label{Eq-m(a)<Phivn}
            m(a) \leq J_\mu\left(v_n\right)=J_\mu\left(u_n\right)+o_n(1).
        \end{equation}
        Combining \eqref{Eq-Phiun<0} and \eqref{Eq-m(a)<Phivn}, {we arrive at}
        $$
        m(a) \leq m\left(a_n\right)+\varepsilon+o_n(1) .
        $$
        Now, let $u \in V(a)$ be such that
        $$
        J_\mu(u) \leq m(a)+\varepsilon \quad \text { and } \quad J_\mu(u)<0 .
        $$
        Define $w_n:={\frac{a_n}{a}} u$, so $w_n \in S\left(a\right)$. Clearly, $\| u\|_{0}<\rho_0$ and $c_n \rightarrow c$ imply $\left\| w_n\right\|_{0}<\rho_0$ for $n$ large enough, thus $w_n \in V\left(a_n\right)$. Moreover, since $J_\mu\left(w_n\right) \rightarrow J_\mu(u)$, it follows that
        $$
        m\left(a_n\right) \leq J_\mu\left(w_n\right)=J_\mu(u)+\left(J_\mu\left(w_n\right)-J_\mu(u)\right) \leq m(a)+\varepsilon+o_n(1).
        $$
    Since $\varepsilon>0$ is arbitrary, it follows that $m\left(a_n\right) \rightarrow m(a)$, completing the proof.
    \end{proof}

{\begin{lemma} \label{5C2} Let $(u_n)\subset V(a)$ be a minimizing sequence with respect to $m(a)$ such that  $u_n \rightharpoonup u_a$ in $X$, \linebreak $u_n(x) \to u_a(x)$ a.e. in $\mathbb{R}^2$ and $u \not=0 $. Then, $u \in S(a)$, $J(u)=m(a)$ and $u_n \to u$ in $X$.
\end{lemma}}

\begin{proof}
    {We aim to show that $w_n:=u_n-u_a \rightarrow 0$ in $X$}, since $u_n(x) \to u(x)$ a.e. in $\mathbb{R}^2$, one has
    \begin{equation}\label{Eq-w2}
        \left|w_n\right|_2^2=\left|u_n\right|_2^2-\left|u_a\right|_2^2+o_n(1)=a^2-\left|u_a\right|_2^2+o_n(1).
    \end{equation}
    Similarly, since  $u_n \rightharpoonup u$ in $X$, one finds
    \begin{equation}\label{Eq-nw2}
        \left\|w_n\right\|_0^2=\left\| u_n\right\|_0^2-\left\| u_a\right\|_0^2+o_n(1).
    \end{equation}
    {Following the same argument as in \eqref{BL-2}-\eqref{BL-4}, we deduce that}
    \begin{equation}\label{Eq-Phibrez} J_\mu\left(u_n\right)=J_\mu\left(w_n\right)+J_\mu\left(u_a\right)+o_n(1).
    \end{equation}
    Now, we claim that
    $$
    \left |w_n\right  |_2^2 \rightarrow 0\quad\,\,\text{as}\quad\,\, n\rightarrow\infty.
    $$
    In order to prove this, let us denote $a_1:=\left|u_a\right|_2>0$. By \eqref{Eq-w2}, if we show that $a_1=a$, then the claim follows. Assume by contradiction that $a_1<a$. In view of \eqref{Eq-w2}, \eqref{Eq-nw2}, for $n$ large enough, we have $\left|w_n\right|_2 \leq a$ and $\left\| w_n\right\|_{0} \leq\left\| u_n\right\|_{0}<\rho_0$. Hence, $w_n \in V\left(\left|w_n\right|_2  \right)$ and $J_\mu\left(w_n\right) \geq m\left(\left|w_n\right|_2 \right)$. {recalling that $J_\mu\left(u_n\right) \rightarrow m(a)$ in \eqref{Eq-Phibrez}, we arrive at}
    \begin{equation}\label{Eq-mcc}
        m(a)=J_\mu\left(w_n\right)+J_\mu\left(u_a\right)+o_n(1) \geq m\left(\left|w_n\right|_2\right)+J_\mu\left(u_a\right)+o_n(1).
    \end{equation}
    Since the map $a \mapsto m(a)$ is continuous (see Lemma \ref{Lem4-m(a)cont}), \eqref{Eq-w2}  gives
    \begin{equation}\label{Eq-mc1}
        m(a) \geq m\left(\sqrt{a^2-a_1^2}\right)+J_\mu\left(u_a\right).
    \end{equation}
    We also have that $u_a \in V\left(a_1\right)$, which implies that $J_\mu\left(u_a\right) \geq m\left(a_1\right)$. If $J_\mu\left(u_a\right)>m\left(a_1\right)$, then it follows from \eqref{Eq-mc1} and Lemma \ref{Lem4-subadd} that
    $$
    m(a)>m\left(\sqrt{a^2-a_1^2}\right)+m\left(a_1\right) \geq m\left(\sqrt{a^2-a_1^2}+a_1\right)=m(a)
    $$
    which is impossible. Hence, we have $J_\mu\left(u_a\right)=m\left(a_1\right)$, that is, $u_a$ is a local minimizer on $V\left(a_1\right)$. Thus, using Lemma \ref{Lem4-subadd} with the strict inequality, we deduce from \eqref{Eq-mc1} that
    $$
    m(a) \geq m\left(\sqrt{a^2-a_1^2}\right)+J_\mu\left(u_a\right)=m\left(\sqrt{a^2-a_1^2}\right)+m\left(a_1\right)>m\left(\sqrt{a^2-a_1^2}+a_1\right)=m(a)
    $$
    which is impossible. Thus,  $ \left|w_n\right|_2^2 \rightarrow 0 $ and from \eqref{Eq-w2} it follows that $\left|u_a\right|_2^2=a^2$.
    It follows immediately, by Lemma \ref{Gagliardo} that
    $$
    |w_n|^{q}_{q} \leq C_q|w_n|^{(1-\beta )q}_{2}\left(\int_{\mathbb{R}^2}\left(\left|(w_n)_x\right|^2+\left|D_x^{-1} (w_n)_y\right|^2\right) d xdy\right)^{\frac{q\beta}{2}}=o_n(1).
    $$
    Finally, from \eqref{Eq-mcc}, we obtain
    $$
    m(a)=J_\mu\left(w_n\right)+J_\mu\left(u_a\right)+o_n(1) \geq \frac{1}{2}\left\|w_n\right\|_0^2+m(a)+o_n(1),
    $$
    which indicates $\left\| w_n\right\|_0 =o_n(1)$. Hence,  $w_n \rightarrow 0$ in $X$ and
    $$u_n\left(x-y_n\right) \rightarrow u_a \not \equiv 0 \text{ in } X.$$
\end{proof}

\begin{lemma}\label{Lem5-vanish}
    For any $a\in\left(0, a_0\right)$, let $(u_n) \subset B_{\rho_0}$ satisfy $\left|u_n\right|_2 \rightarrow a$ and $J_\mu\left(u_n\right) \rightarrow m(a)$. Then, for each $R>0$ fixed, there exist  $\beta_1>0$ and a sequence $(x_n,y_n) \subset \mathbb{R}^2$ such that
    \begin{equation}\label{Eq-beta1}
        \int_{B_R\left(x_n, y_n\right)}\left|u_n\right|^2 d x \geq \beta_1>0.
    \end{equation}
\end{lemma}
\begin{proof}
    We assume by contradiction that \eqref{Eq-beta1} does not hold. Since $(u_n)  \subset B_{\rho_0}$ and $\left|u_n\right|_2 \rightarrow a$, the sequence $(u_n) $ is bounded in $X$. By {
    \cite[Lemma I.1]{LionsCCPLocal2}}, {we know that} $\left|u_n\right|_p \rightarrow 0$ as $n \rightarrow +\infty$ up to a translation. Thus,
    {$$
    J_\mu\left(u_n\right) =\frac{1}{2}\left\| u_n\right\|_0^2+o_n(1)\geq o_n(1),
    $$}
{contradicting the fact that $m(a)<0$ from Lemma \ref{Lem5-m(a)}. Hence, the result follows.}
\end{proof}
Next, we prove the first part of Theorem \ref{Th3}.

\noindent \textbf{Proof of the first part of Theorem \ref{Th3}:}
{Let $ (u_n)\subset V(a) $  be a minimizing sequence for $m(a)$
By Lemma \ref{Lem5-vanish}, there exists a sequence $(x_n,y_n) \subset \mathbb{R}^2$ such that
\begin{equation}\label{Eq-weaku}
    u_n\left(x+x_n,y+y_n\right) \rightharpoonup u_a \neq 0 \quad \text { in } X .
\end{equation}
From Lemma \ref{5C2}, we deduce that, $J({u}_a)=m(a)$ and $u_n\left(x+x_n,y+y_n\right) \to {u}_a$ in $X$.
 By Lemma \ref{Lem4-ground}, this minimizer $u_a$ is a normalized ground state solution for problem \eqref{Equation},  and any normalized ground state solution for problem \eqref{Equation}  belongs to $V(a)$.

{Now, we are ready to prove the second part of Theorem \ref{Th3}. More precisely, we are going to show that there exists a sequence $(a_n) \subset (0, a_0)$ with $a_n \to 0$ as $n \to +\infty$, such that for each $a = a_n$, the problem \eqref{Equation} admits a second solution with positive energy. This completes the proof of Theorem \ref{Th3}. To this end, we first present the following lemma, which plays a crucial role in showing that the second solution has positive energy.}

\begin{lemma}\label{Mp}
    For any $\mu>0$ and $a \in\left(0, a_0\right)$, there exists $\kappa_{\mu, a}>0$ such that
    \begin{equation}\label{Eq5-3}
        M(a):=\inf _{\gamma \in \Gamma_{ a}} \max _{t \in[0,1]} J_\mu(\gamma(t)) \geq \kappa_{\mu,a}>\sup _{\gamma \in \Gamma_{ a}} \max \left\{J_\mu(\gamma(0)), J_\mu(\gamma(1))\right\}
    \end{equation}
    where
    \begin{equation}\label{Eq5-Gamma}
        \begin{aligned}
            &\Gamma_{ a}=\left\{\gamma \in \mathcal{C}([0,1], {S}(a): \gamma(0)=u_a, J_\mu(\gamma(1))<2 m(a)\right\} .
        \end{aligned}
    \end{equation}
\end{lemma}
\begin{proof}
    {Define  $\kappa_{\mu,a}:=\inf _{u \in \partial V(a)} J_\mu(u)$. By \eqref{eq-m<0}, $\kappa_{\mu,a}>0$. For any $\gamma \in \Gamma_{ a}$, since $\gamma(0)=u_a \in V(a) \backslash(\partial V(a))$ and $J_\mu(\gamma(1))<2 m(a)$, we have  $\gamma(1) \notin V(a)$. By the continuity of $\gamma(t)$ on $[0,1]$, there exists a $t_0 \in(0,1)$ such that $\gamma\left(t_0\right) \in \partial V(a)$, and so
 $$
 \max _{t \in[0,1]} J_\mu(\gamma(t)) \geq  J_\mu(\gamma(t_0))\geq\kappa_{\mu,a}.
 $$
 Since $J_\mu(\gamma(0)) = J_\mu(u_a) = m(a) < 0$ and $J_\mu(\gamma(1)) < 2 m(a) < 0$, we have $\kappa_{\mu,a} > \max \{ J_\mu(\gamma(0)), J_\mu(\gamma(1)) \}$, proving \eqref{Eq5-3}.}
\end{proof}

\noindent \textbf{Proof of the second part of Theorem \ref{Th3}:}
        By Lemma \ref{Mp} and \cite[Theorem 2.5]{Chen2024Another}, there exists a (PS) sequence $(u_n) \subset S(a)$ associated with the mountain pass level $M(a)>0$, that is,
        \[
        J_\mu(u_n) \to M(a) \quad \text{and} \quad J_\mu'|_{S(a)}(u_n) \to 0 \quad \text{as} \quad n \to +\infty.
        \]
        It is easy to verify that the conclusion of Lemma~\ref{profile} remains valid for the functional $J_\mu$. Consequently, there exist $\ell \in \mathbb{N}$ and sequences $\left(\widetilde{u}_i\right)_{i=0}^{\ell} \subset X$, $\left((x_n^i,y_n^i)\right)_{i=0}^{\ell} \subset \mathbb{R}^2$ with $(x_n^0,y_n^0)=(0,0)$, such that
        \[
        (x_n^i - x_n^j)^2 + (y_n^i - y_n^j)^2 \to +\infty \quad \text{as} \quad n \to +\infty \quad \text{for} \quad i \neq j,
        \]
        and, up to a subsequence, {the following hold for each $i\in \{0,1,\ldots,\ell\}$},
        \begin{equation}\label{eq:com122}
            \begin{aligned}
                u_n(\cdot - x_n^i, \cdot - y_n^i) \rightharpoonup \widetilde{u}_i \text{ in } X, \quad
                J_\mu'|_{S(b_i)}(\widetilde{u}_i) = 0, \quad \text{where } b_i = |\widetilde{u}_i|_2, \text{ and } a^2 = \sum_{i=0}^{\ell} b_i^2,
            \end{aligned}
        \end{equation}
        \begin{equation}\label{eq:com11}
            \left\| u_n - \sum_{i=0}^{\ell} \widetilde{u}_i(\cdot + x_n^i, \cdot + y_n^i) \right\| \to 0 \quad \text{as } n \to \infty,
        \end{equation}
        \begin{equation}\label{eq:com13}
            \sum_{i=0}^{\ell} J_\mu(\widetilde{u}_i) = \lim_{n \to \infty} J_\mu(u_n) = M(a).
        \end{equation}
        Since $M(a) > 0$, there exists some $i_0 \in \{0,1,\ldots,\ell\}$ such that
        \[
        J_\mu(\widetilde{u}_{i_0}) > 0, \quad J_\mu'|_{S(b_{i_0})}(\widetilde{u}_{i_0}) = 0, \quad \text{with } b_{i_0} = |\widetilde{u}_{i_0}|_2 \leq a < a_0.
        \]
        Thus, problem \eqref{Equation} with $a = b_{i_0}$ has a second solution with positive energy.

{ Define $\tilde{a}_{2} = \min\{\frac{1}{2}, b_{i_0}\}$. Repeating the argument, there exists $a_2 \in (0, \tilde{a}_{2}]$ such that problem \eqref{Equation} with $a = a_2$ admits a second solution with positive energy. Inductively, or each $n \geq 2$, set $\tilde{a}_{n+1}=\min\{\frac{1}{n+1}, a_n\}$, then there exists $a_{n+1} \in (0, \tilde{a}_{n+1}]$ such that problem \eqref{Equation} with $a = a_{n+1}$ admits a second solution with positive energy. Thus, we obtain a sequence $(a_n) \subset (0, a_0)$ with $a_n \to 0$ as $n \to +\infty$, where each $a_n$ satisfies the desired property.}

\subsection*{Conflict of interest}

On behalf of all authors, the corresponding author states that there is no conflict of interest.

\subsection*{Ethics approval}
 Not applicable.

\subsection*{Data Availability Statements}
Data sharing not applicable to this article as no datasets were generated or analysed during the current study.

\subsection*{Acknowledgements}
C.O. Alves is supported by CNPq/Brazil 307045/2021-8 and Projeto Universal FAPESQ-PB 3031/2021. C. Ji is supported by National Natural Science Foundation of China (No. 12171152).


\begin{thebibliography}{99}

\bibitem{A21} C.O. Alves, {\it On existence of multiple normalized solutions to a class of elliptic problems in whole $\mathbb R^N,$}
Z. Angew. Math. Phys., 73 (2022), no. 3, Paper No. 97, 17 pp.

\bibitem{CCM} C.O. Alves,  C. Ji, O.H. Miyagaki, {\it Normalized solutions for a Schr\"{o}dinger equation with critical growth in $\mathbb{R}^{N}$}, Calc. Var. Partial Differential Equations, 61 (2022), no. 1, Paper No. 18, 24 pp.

\bibitem{AJ21} C.O. Alves, C. Ji, {\it Normalized solutions for the Schr\"odinger equations with $L^2$-subcritical growth and different type of potentials},
J. Geom. Anal., 32 (2022), no. 5, Paper No. 165, 25 pp.

\bibitem{AS2} Alves, C.O., Shen, L.: On existence of normalized solutions to a class of elliptic problems with $L^2-$supercritical grow. To appear in J. Differential Equations, https://doi.org/10.1016/j.jde.2025.02.059.

\bibitem{AlvesThin} C.O. Alves and N.V. Thin, {\it On existence of multiple normalized solutions to a class of elliptic problems in whole $\mathbb{R}^N$ via Lusternik-Schnirelmann category},
SIAM J. Math. Anal., 55 (2023), no. 2, 1264--1283.


 \bibitem{AJ} C.O. Alves, C. Ji,  \emph{Existence and concentration of nontrivial solitary waves for a generalized Kadomtsev-Petviashvili equation in $\mathbb{R}^2$}. J. Differential Equations  \textbf{368} (2023), 141-172.


\bibitem{AM} C.O. Alves, O.H. Miyagaki, \emph{Existence, regularity and concentration phenomenon of nontrivial solitary waves for a class of generalized variable coefficient Kadomtsev--Petviashvili equation}. J. Math. Phys. \textbf{58} (2017), 081503.



\bibitem{AMP} C.O. Alves, O.H. Miyagaki, A. Pomponio, \emph{Solitary waves for a class of generalized Kadomtsev--Petviashvili
    equation in $\mathbb{R}^N$ with positive and zero mass}. J. Math. Anal. Appl. \textbf{477} (2019), 523--535.

\bibitem{AGM} C.O. Alves, G.M. Figueiredo, M. Montenegro, \emph{On the energy of the  ground state solution for a generalized Kadomtsev--Petviashvili equation}, Preprint.




\bibitem{Bartschmolle}  T. Bartsch, R. Molle, M. Rizzi, G. Verzini, {\it Normalized solutions of mass supercritical Schr\"odinger  equations with potential},  Comm. Partial Differential Equations, 46 (2021), no. 9, 1729--1756.

\bibitem{BartschSaove} T. Bartsch, N. Soave, {\it Multiple normalized solutions for a competing system of Schr\"odinger equations}, Calc. Var. Partial Differential Equations, 58 (2019), no. 1, Paper No. 22, 24 pp.

\bibitem{ThomaseNicola} T. Bartsch, N. Soave, {\it A natural constraint approach to normalized solutions of nonlinear Schr\"odinger equations and systems}, J. Func. Anal. 272 (2017), 4998-5037.

\bibitem{valerio} T. Bartsch, S. De Valeriola, {\it Normalized solutions of nonlinear Schr\"odinger equations},  Arch.
Math., 100 (2013), 75-83.

\bibitem{BellazziniJeanjeanLuo} J. Bellazzini, L. Jeanjean, T. Luo, {\it Existence and instability of standing waves with prescribed norm for a class of
    Schr\"odinger-Poisson equations}, Proc. Lond. Math. Soc. (3), 107 (2013), no. 2, 303--39.

\bibitem{BIN} O. V. Besov, V. P. Il'in, S.M. Nikolski, \emph{Integral Representations of Functions and Imbedding Theorems, Volume I}. Wiley, New York, (1978).

\bibitem{Bartosz} B. Bieganowski, J. Mederski, {\it Normalized ground states of the nonlinear Schr\"odinger equation with at least mass critical growth},
J. Funct. Anal., 280 (2021), no. 11, Paper No. 108989, 26 pp.

\bibitem{B} J. Bourgain, \emph{On the Cauchy problem for the Kadomtsev-Petviashvili equation}. Geom. Funct. Anal. \textbf{3} (1993), 315--341.


\bibitem{Cazenave} T. Cazenave, {\it Semilinear Schr\"odinger equations,} Courant Lecture Notes in Mathematics {\bf 10 }(New York University, Courant Institute of Mathematical Sciences, New York; American Mathematical Society, Providence, RI, 2003, 323 pp. ISBN: 0-8218-3399-5.


\bibitem{CL} T. Cazenave,  P.L. Lions, \emph{ Orbital stability of standing waves for some nonlinear Schr\"odinger equations},  Comm.
Math. Phys.  \textbf{85}(4) (1982), 549--561.


\bibitem{Chen2024Another}
S.T. Chen, X.H. Tang,
\textit{Another look at Schr\"odinger equations with prescribed mass},
J. Differential Equations \textbf{386} (2024), 435--479.

\bibitem{DS0}A. De Bouard, J.-C. Saut, \emph{Sur les ondes solitaires des \'{e}quations de Kadomtsev-Petviashvili}. C. R. Acad. Sci. Paris, \textbf{320} (1995), 315--318.


\bibitem{DS}A. De Bouard, J.-C. Saut, \emph{Solitary waves of generalized Kadomtsev-Petviashvili equations}. Ann. Inst. Henri Poincar\'{e}, Anal. Non Lin\'{e}aire \textbf{14} (1997),  211--236.

\bibitem{DF} M. del Pino, P.L. Felmer, \emph{Local Mountain Pass for semilinear elliptic problems in unbounded domains}. Cal. Var. Partial Differential Equations  \textbf{4} (1996), 121-137.

\bibitem{F} A. Faminskii, \emph{The Cauchy problem for Kadomtsev-Petviashvili equation}. Russ. Math. Surv. \textbf{5} (1990), 203--204.



\bibitem{FM} G. Figueiredo, M. Montenegro, \emph{Multiple solitary waves for a generalized Kadomtsev--Petviashvili equation with a potential}. J. Differential Equations \textbf{308} (2022), 40--56.





\bibitem{SYM}
M. Grillakis, J. Shatah, W. Strauss, \emph{Stability theory of solitary waves in the presence of symmetry: I}, J. Funct. Anal.\textbf{74} (1987), 160--197.


\bibitem{GW} F. G\"{u}ng\"{o}r, P. Winternitz, \emph{Generalized Kadomtsev-Petviashvili equation with an infinite-dimensional symmetry
algebra}. J. Math. Anal. Appl. \textbf{276} (2002), 314--328




\bibitem{HNS} N. Hayashi, P.I. Naumkin, J.-C. Saut, \emph{Asymptotics for large time of global solutions to the generalized KadomtsevPetviashvili equation}. Commun. Math. Phys.  \textbf{201} (1999), 577--590.

\bibitem{HZ} X.M. He, W.M. Zou,  \emph{Nontrivial solitary waves to the generalized Kadomtsev-Petviashvili equations}.  Appl. Math. Comput. \textbf{197} (2008), 858--863.


\bibitem{Jun} J. Hirata, K. Tanaka, {\it Nonlinear scalar field equations with $L^2$ constraint: mountain pass and symmetric mountain pass approaches}, Adv. Nonlinear Stud., 19 (2019), no. 2, 263--290.



\bibitem{IM} P. Isaza, J. Mejia,  \emph{Local and global Cauchy problem for the Kadomtsev-Petviashvili (KP-II) equation in Sobolev spaces of negative indices}. Commun. Partial Differ. Equ. \textbf{26} (2001), 1027--1054.

\bibitem{Jean} L. Jeanjean, {\it Existence of solutions with prescribed norm for semilinear elliptic equations}, Nonlinear Anal. \textbf{28} (1997), 1633--1659.

\bibitem{JeanjeanLu} L. Jeanjean, S.S. Lu, {\it Nonradial normalized solutions for nonlinear scalar field equations}, Nonlinearity,
32 (2019), no. 12, 4942--4966.

\bibitem{JeanjeanLu2020} L. Jeanjean, S.S. Lu,  {\it A mass supercritical problem revisited},  Calc.
Var. Partial Differential Equations, 59 (2020), no. 5, Paper No. 174, 43 pp.

\bibitem{JeanjenaLu2} L. Jeanjean, S.S. Lu, {\it On global minimizers for a mass constrained problem},
Calc. Var. Partial Differential Equations, 61 (2022), no. 6, Paper No. 214, 18 pp

\bibitem{JeanjeanLe} L. Jeanjean, T.T. Le, {\it Multiple normalized solutions for a Sobolev critical Schr\"odinger  equations}, Math. Ann.,
384 (2022), no. 1--2, 101--134.

\bibitem{JeanjeanJendrejLeVisciglia} L. Jeanjean, J. Jendrej, T.T. Le, N. Visciglia, {\it Orbital stability of ground states for a Sobolev critical
    Schr\"odinger equation}, J. Math. Pures Appl. (9), 164 (2022), 158--179.

\bibitem{KP1} B.B. Kadomtsev, V.I. Petviashvili, \emph{On the stability of solitary waves in weakly dispersing media}. Sov. Phys. Dokl. \textbf{15} (1970), 539--541.


\bibitem{LS} Z.P. Liang, J.B. Su, \emph{Existence of solitary waves to a generalized Kadomtsev--Petviashvili equation}. Acta Math. Sci. \textbf{32} (2012), 1149--1156.



    \bibitem{LionsCCPLocal2}
P.-L. Lions,
\textit{The concentration-compactness principle in the calculus of variations. The locally compact case, part 2},
Ann. Inst. H. Poincar{\'e} C Anal. Non Lin{\'e}aire \textbf{1} (1984), 223--283.

\bibitem{PetviashviliYankov}
V. Petviashvili, V. Yan'kov,
\textit{Solitons and turbulence},
in B. B. Kadomtsev (ed.), Rev. Plasma Phys. \textbf{14} (1989), 1--62.


\bibitem{S} J.-C. Saut,  \emph{Recent results on the generalized Kadomtsev-Petviashvili equations}. Acta Appl. Math. \textbf{39} (1995), 1477--1487.

\bibitem{Shibata}
M. Shibata, \emph{Stable standing waves of nonlinear Schrodinger equations with a general nonlinear term},
Manuscripta Math. \textbf{143} (2014), 221--237.


\bibitem{Nicola1} N. Soave, {\it Normalized ground states for the NLS equation with combined nonlinearities: the Sobolev critical case}, J. Funct. Anal.,
279 (2020), no. 6, 108610, 43 pp.

\bibitem{SoavenonC}
N. Soave, \textit{Normalized ground states for the NLS equation with combined nonlinearities},
J. Differential Equations \textbf{269} (2020), 6941--6987.

\bibitem{SZ} A. Szulkin, W.M. Zou, \emph{Homoclinic orbit for asymptotically linear Hamiltonian systems}. J. Funct. Anal. \textbf{187} (2001), 25--41.

\bibitem{TG} B. Tian, Y.T. Gao,  \emph{Solutions of a variable-coefficient Kadomtsev-Petviashvili equation via computer algebra}.  Appl. Math. Comput. \textbf{84} (1997), 125--130.



\bibitem{T} N. Tzvetkov, \emph{Global low regularity solutions for Kadomtsev-Petviashvili equations}. Differ. Integral Equ. \textbf{13} (2000), 1289--1320.


\bibitem{WW} Z.-Q. Wang, M. Willem, \emph{A multiplicity result for the generalized Kadomtsev-Petviashvili equation}. Topol. Methods Nonlinear Anal. \textbf{7} (1996), 261--270.

\bibitem{Willem} M. Willem, \emph{ Minimax Theorems}.  Birkh\"{a}user, 1996.

\bibitem{XWD} J. Xu, Z. Wei, Y. Ding, \emph{Stationary solutions for a generalized Kadomtsev-Petviashvili equation in bounded domain}. Electron. J. Qual. Theory Differ. Equ. \textbf{68} (2012), 1--18.

\bibitem{Xu} B. Xuan, \emph{Nontrivial solitary waves of GKP equation in multi-dimensional spaces}. Rev. Colombiana Mat. \textbf{37} (2003), 11--23.

\bibitem{ZXM} Y. Zhang, Y. Xu, K. Ma, \emph{New type of a generalized variable-coefficient Kadomtsev-Petviashvili equation with self-consistent sources and its Grammian-type solutions}. Commun. Nonlinear Sci. Numer. Simul. \textbf{37} (2016), 77--89.

\bibitem{Z} W.M. Zou, \emph{Solitary waves of the generalized Kadomtsev-Petviashvili equations}. Appl. Math. Lett. \textbf{15} (2002), 35--39.









\end{thebibliography}
\end{document}